\documentclass[11pt]{article}
\usepackage{amsmath, amssymb, amsfonts, amstext, amsthm, textcomp, enumerate, mathtools, color}
\usepackage[mathscr]{euscript}
\usepackage{soul, cancel}
\newtheorem{thm}{Theorem}[section]
\newtheorem{lem}[thm]{Lemma}
\newtheorem{cor}[thm]{Corollary}
\newtheorem{pro}[thm]{Proposition}
\theoremstyle{definition}
\newtheorem{defn}[thm]{Definition}
\newtheorem{rem}[thm]{Remark}
\newtheorem{ex}[thm]{Example}

\begin{document}
\begin{center}
{\bf \large Karamardian Matrices: An Analogue of $Q$-Matrices}\\
\vspace{3cm}
{\bf K.C. Sivakumar$^{*}$, P. Sushmitha$^{*}$ and M. Wendler$^{**}$}\\
\vspace{.5cm}
$^{*}$Department of Mathematics\\
Indian Institute of Technology Madras\\
Chennai 600 036\\
India \\
\vspace{.5cm}
$^{**}$Department of Mathematics and Statistics\\
Colorado Mesa University\\
Grand Junction, CO 81501,\\ U.S.A. 
\end{center}

\begin{abstract}
A real square matrix $A$ is called a $Q$-matrix if the linear complementarity problem LCP$(A,q)$ has a solution for all $q \in \mathbb{R}^n$. This means that for every vector $q$ there exists a vector $x$ such that $x \geq 0, y=Ax+q\geq 0$ and $x^Ty=0$. A well known result of Karamardian states that if the problems LCP$(A,0)$ and LCP$(A,d)$ for some $d\in \mathbb{R}^n, d >0$ have only the zero solution, then $A$ is a $Q$-matrix. Upon {\it relaxing} the requirement on the vectors $d$ and $y$ so that the vector $y$ belongs translation of the nonnegative orthant by the null space of $A^T$, $d$ to belong to its interior, and imposing the {\it additional} condition on the solution vector $x$ to be in the intersection of the range space of $A$ with the nonnegative orthant, in the two problems as above, the authors introduce a new class of matrices called Karamardian matrices, wherein these two modified problems have only zero as a solution. In this article, a systematic treatment of matrices is undertaken. Among other things, it is shown how Karamardian matrices have properties that are analogous to those of $Q$-matrices. A subclass of a recently introduced notion of $P_{\#}$-matrices is shown to possess the Karamardian property, and for this reason we undertake a thorough study of $P_{\#}$-matrices and make some fundamental contributions.
\end{abstract}

{\bf AMS Subject Classification (2010):} 15A09, 90C33

{\bf Keywords:} Linear complementarity problem; $P$-matrix; $Q$-matrix; $Z$-matrix; group inverse; range monotonicity; nonnegativity.

\newpage

\section{Introduction and Motivation}
Let $\mathbb{R}^{n\times n}$ denote the set of all real matrices of order $n \times n$. $\mathbb{R}^{n\times 1}$ will be denoted by $\mathbb{R}^{n}$. We say that a vector $x=(x_1,x_2,...,x_n)^T \in \mathbb{R}^n$ is nonnegative and denote it by $x \geq 0$ if, and only if, $x_i\geq 0$ for all $i\in \{1,2,...,n\}$. A vector $x$ is said to be positive if, and only if, $x_i>0$ for all $i\in \{1,2,...,n\}$. We denote it by $x>0$. A real matrix $A$ is called nonnegative if all its entries are nonnegative and this will be denoted by $A \geq 0$. By $A>0$, we mean that all the entries of $A$ are positive.

Let us recall the central notion of this article. Given $A\in \mathbb{R}^{n\times n},$ and $q \in \mathbb{R}^n$, the {\it linear complementarity problem}, denoted by LCP$(A,q)$ is to find $x \in \mathbb{R}^n$ such that 
\begin{center}
$x \geq 0, ~y=Ax+q \geq 0$ and $x^Ty =0.$
\end{center}
An LCP is a special instance of a variational inequality problem and arises in a wide range of applications, including in linear programming and bimatrix games. We refer the reader to the excellent book \cite{cps} for more details.

Recall that a real square matrix $A$ is said to be a {\it $P$-matrix}, if all its principal minors are positive. It follows at once, that $A$ is a $P$-matrix if, and only if, $A^{-1}$ (exists and) is a $P$-matrix. It is well known that $A$ is a $P$-matrix if, and only if, $A$ does not reverse the sign of any nonzero vector, viz., the following implication holds \cite{fp}:
\begin{center}
$x_i(Ax)_i \leq 0$ for all $i = 1,2, \ldots, n \quad \Longrightarrow \quad x=0.$ 
\end{center}
Perhaps the most important characterization for a $P$-matrix is in terms of the linear complementarity problem:  $A$ is a $P$-matrix if, and only if, the linear complementarity problem LCP$(A,q)$ has a unique solution for all $q \in \mathbb{R}^n$. A related class of matrices is recalled, next: $A$ is called a {\it $Q$-matrix} if LCP$(A,q)$ has a solution for all $q \in \mathbb{R}^n$. Clearly, any $P$-matrix is a $Q$-matrix. Note, however that the converse is not true. The square matrix all of whose entries equals $1$ is an example of a $Q$-matrix which is not a $P$-matrix. One way to prove that such a matrix is indeed a $Q$-matrix, is to use a special case of a result of Karamardian \cite{kar}, which is stated below. This is the most frequently used result to prove that a given matrix is a $Q$-matrix. It is also important to underscore the fact that this motivates the class of matrices that are primarily studied in this article (see Definition \ref{karamdef}). 

\begin{thm}\label{karthm}
Suppose that LCP$(A,0)$ and LCP$(A,d)$ have a unique solution, for some $d >0$. Then LCP$(A,q)$ has a solution for all $q \in \mathbb{R}^n$. 
\end{thm}

Before we discuss a special class of $Q$-matrices, we would like to state a result that is used in some of the numerical examples. It is the statement that a nonnegative matrix $A$ is a $Q$-matrix if, and only if, the diagonal entries of $A$ are positive. Also, a matrix with a nonpositive row (and hence a nonpositive matrix) is not a $Q$-matrix. For instance, let  the $i^{th}$ row of $A$ be nonpositive and $e^i$ be the vector whose $i$th coordinate is one, while all its other entries are zero. Then, LCP$(A,-e^i)$ has no solution. It is also important to note that the inverse of an invertible $Q$-matrix is a $Q$-matrix. Let us now turn our focus on a class of matrices for which every $Q$-matrix is also a $P$-matrix. A matrix $A$ is called a {\it $Z$-matrix} if its off-diagonal entries are nonpositive. A reformulation of this notion, useful in many proofs to follow, is the following: $A$ is a $Z$-matrix if, and only if, 
\begin{center}
$x \geq 0, ~y \geq 0$ \textit{and} $x^Ty =0 \Longrightarrow (Ax)^Ty \leq 0.$
\end{center}
Let $A$ be a $Z$-matrix. If $A$ can be written as $A=sI-B$, where $B$ is a nonnegative matrix and $s \geq \rho(B)$, then $A$ is called as an {\it $M$-matrix}. An $M$-matrix $A$ (with the representation as above), is nonsingular if $s > \rho(B)$ and singular if $s= \rho(B)$. It is known that a nonsingular $M$-matrix $A$ has the property that all the entries of $A^{-1}$ are nonnegative, i.e., $A$ is inverse nonnegative. In this connection, the following result is quite well known:

\begin{thm}(\cite{berpl})\label{zq}
let $A$ be a $Z$-matrix. Then the following statements are equivalent:\\
$(a)~ A$ is a $P$-matrix.\\
$(b)~ A$ is a $Q$-matrix.\\
$(c)~ A$ is an invertible $M$-matrix.\\
$(d)~A$ is monotone.
\end{thm}

To motivate the next part of this introductory section, we need the concept of generalized inverses. For $A\in \mathbb{R}^{n\times n},$ let $R(A)$ and $N(A)$ denote the range space and the null space of the matrix $A$. Given $A\in \mathbb{R}^{n \times n},$ consider the following three equations for $X\in \mathbb{R}^{n\times n}:$ $A$ satisfying $AXA=A, XAX=X$ and $AX=XA$. Such an $X$ need not exist; however if it exists, then it is unique, is called the {\it group inverse of $A$} and is denoted by $A^{\#}$. The following are two well known characterizations for $A^{\#}$ to exist: the ranks of $A$ and $A^2$ are the same; $R(A)$ and $N(A)$ are complementary subspaces of $\mathbb{R}^n$. For the matrix $A=\begin{pmatrix}
0 & 1 \\ 0 & 0
\end{pmatrix},$
$A^{\#}$ does not exist, whereas if $A$ is the $n \times n$ matrix each of whose entries equals $1$, then one may verify that $A^{\#}=\frac{1}{n^2}A$. The following properties of the group inverse may be frequently used in some of the proofs: $x \in R(A)$ if, and only if, $x=AA^{\#}x$; $R(A)=R(A^{\#})$ and $N(A)=N(A^{\#})$. There is a much more famous and most widely used generalized inverse, viz., the Moore-Penrose inverse. Recall that for any $A\in \mathbb{R}^{m\times n}$ there exists a unique $X\in \mathbb{R}^{n\times m}$ satisfying the equations: $AXA=A, XAX=X, (AX)^T=AX$ and $(XA)^T=XA$. Such a matrix $X$ is called the {\it Moore-Penrose inverse} of $A$ and is denoted by $A^{\dagger}$. While the group inverse does not exist for all (square) matrices, the Moore-Penrose inverse exists for every matrix. For those square matrices for which the group inverse exists, it is not necessary that the group inverse and the Moore-Penrose inverse coincide. There is a special class of matrices for which this holds, however. Let us recall that a matrix $A$ is called {\it range-symmetric} if $R(A)=R(A^T)$. It is well known that for such matrices one has $A^{\#}=A^{\dagger}$. The following properties of the two generalized inverses above may be useful later. $AA^{\dagger}$ is the orthogonal projection on $R(A)$ (while $A^{\dagger}A$ is the orthogonal projection in $R(A^T)$); $I-AA^{\dagger}$ is the orthogonal projection on $N(A^T)$ (while $I-A^{\dagger}A$ is the orthogonal projection on $N(A)$); $AA^{\#}=A^{\#}A$ is a projection on $R(A)$ (while $I-AA^{\#}$ is a projection on $N(A)$), possibly  nonorthogonal. In particular, all these are idempotent matrices. We refer the reader to \cite{ben} for a proof of the statements above and for more details on generalized inverses. 

A real square matrix $A$ is called {\it inverse positive}, if $A$ is invertible and all the entries of $A^{-1}$ are nonnegative. A possibly rectangular matrix $A$ is referred to as {\it monotone} if $Ax \geq 0$ implies $x \geq 0$. It is easy to prove that a square matrix $A$ is monotone if, and only if, $A$ is inverse positive. It is useful to observe that if $A$ is inverse positive, then $A^{-1}x \geq 0$, whenever $x \geq 0$. It is also known from \cite{mang} that, a rectangular matrix $A$ is monotone if, and only if, $A$ has a nonnegative left-inverse. There are several notions that have been proposed as extensions of monotonicity. Let us recall the one that is pertinent to the discussion here. A square matrix $A$ is called {\it range monotone} if one has the following implication: $Ax \geq 0, ~x \in R(A)$ implies $x \geq 0$. It is easy to see that this implication is an extension of the monotonicity notion, above. It is known that $A$ is range monotone if, and only if, $A^{\#}$ exists and is nonnegative on $R(A)$. This means that, if $A$ is range monotone, $x \geq 0$ and $x\in R(A)$, then $A^{\#}x \geq 0$. It is now clear that this condition generalizes the corresponding condition for inverse positive matrices, stated earlier. For more details, we refer to \cite{berpl}.

In connection with singular $M$-matrices, the notion of ``property $c$'' was introduced in \cite{pl}. We recall this next. A matrix $T$ is said to be {\it semi-convergent} if $\displaystyle \lim_{n\rightarrow \infty}T^n$ exists. Note that the said limit exists if, and only if, $\displaystyle \lim_{n\rightarrow \infty}\Vert T^n \Vert$ exists, for any matrix norm. An $M$-matrix $A$ is said to have {\it ``property $c$''} if it can be written as $A=sI-B$, where $s>0$, $B\geq 0$ and $B/s$ is semi-convergent \cite{pl}. Any nonsingular $M$-matrix has ``property $c$.'' This is due to the fact that such a matrix $A$ could be written as $A=sI-B$, where $B \geq 0$ with $s > \rho(B)$ so that $B/s$ converges to the zero matrix. The matrix $A=\begin{pmatrix}
\ \ 1 & -1 \\ -1 & \ \ 1
\end{pmatrix}$ is an example of a (non-invertible) $M$-matrix with ``property $c$" for the fact that one has $A=sI-B$, where for $s>1$, $B=\begin{pmatrix}
s-1 & 1 \\ 1 & s-1
\end{pmatrix}\geq 0$ and $B/s$ is semi-convergent. The matrix $A=\begin{pmatrix}
0 & -1 \\ 0 & \ \ 0
\end{pmatrix}$ is not an $M$-matrix with ``property $c$." This is due to the reason that, if $A=sI-B$ with $s>0$ and $B \geq 0$, then $B/s=\begin{pmatrix}
1 & 1/s \\ 0 & 1
\end{pmatrix},$ so that $(B/s)^n=\begin{pmatrix}
1 & n/s \\ 0 & 1
\end{pmatrix},$ which is not semi-convergent. 

A rather distinguished class of $M$-matrices with ``property $c$'' called singular irreducible $M$-matrices are discussed next. A matrix $A\in \mathbb{R}^{n\times n}$ is said to be {\it reducible} if it is permutationally similar to a matrix of the form $\begin{pmatrix}
B&0\\C&D
\end{pmatrix}$, where $B$ and $D$ are square matrices, or if $n=1$ and $A=0$. Otherwise, $A$ is said to be {\it irreducible}. Singular irreducible $M$-matrices have been shown to have many applications in numerical methods for systems of linear equations. A prominent theorem that connects singular irreducible $M$-matrices and $M$-matrices with ``property $c$'' is the following:

\begin{thm}(\cite[Theorem 4.16]{berpl})\label{singirr}
Let $A$ be a singular irreducible $M$-matrix of order $n$. Then \\
$(a)~A$ has rank $n-1$. \\
$(b)$ There exists a vector $x>0$ such that $Ax=0$.\\
$(c)~A$ has ``property $c$''.\\
$(d)$ Every principal submatrix of $A$ other than $A$ itself is a nonsingular $M$-matrix.\\
$(e)~A$ is almost monotone, i.e., $Ax\geq 0 \Longrightarrow Ax=0$.
\end{thm}


In this article, we introduce a new class of matrices in the context of the linear complementarity problem, called Karamardian matrices. Since these matrices also arise from what are called $P_{\#}$-matrices (which were introduced and studied briefly in earlier works) we first undertake a detailed study of $P_{\#}$-matrices. Our endeavour is to demonstrate that these matrix classes possess a variety of positivity properties and also to show how they are related to a number of matrix positivity classes. 

Let us present an outline of the contents of this article. In the next section, we revisit the notion of $P_{\#}$-matrices and consider many of the fundamental aspects of this class. Conditions for a rank one matrix to be a $P_{\#}$-matrix are considered (Lemma \ref{rankonephash}). Various comparisons are made with $P$-matrices and $P_0$-matrices, to show how $P_{\#}$-matrices are similar or different from these classes. These appear in Remark \ref{pnotphash}, Example \ref{diag} and Remark \ref{aplusep}. We outline a procedure for constructing $P_{\#}$-matrices whose leading principal submatrix is a $P$-matrix (Proposition \ref{algo2}). In Section \ref{karamardian}, the notion of a Karamardian matrix is introduced. Some of the basic properties and results for matrices belonging to this class are presented in Subsection \ref{karamardian}. We show how these matrices can be thought of as analogues of $Q$-matrices. Relationship with range monotone matrices are presented in Subsection \ref{monotonicity and kar}, whereas the relationship with other matrix classes is presented in \ref{other matrix classes}. A complete classification of $2\times 2$ Karamardian matrices is presented in Subsection \ref{order2kar}.

\section{$P_{\#}$-Matrices}\label{secphash}
Recently, a certain extension of $P$-matrices called $P_{\#}$-matrices was proposed in \cite{rajkcs} (as an analogue of what are called $P_{\dagger}$-matrices). Only a brief study was undertaken there. A further consideration was made in \cite{ijkcs}, where relationships with certain generalizations of $M$-matrices were proved (see Theorem \ref{necphash} to follow, for instance). As $P_{\#}$-matrices serve as examples for Karamardian matrices (the central objects of study), we now make a detailed study of $P_{\#}$-matrices. Let us first recall its definition. 

\begin{defn}(\cite[Definition 5.1]{rajkcs})\label{phash}
$A \in \mathbb{R}^{n\times n}$ is said to be a $P_{\#}$-matrix, if the following implication holds: 
\begin{center}
$x\in R(A),\ x_i(Ax)_i\leq 0$ for all $i = 1,2,\ldots, n \quad \Longrightarrow \quad x=0$.
\end{center}
\end{defn}

Using the Hadamard entrywise product, one could write the implication above as $$x \in R(A),~x*Ax \leq 0 \quad \Longrightarrow \quad x=0.$$ Recall that $*$ is defined by: if $u,v \in \mathbb{R}^n$, then $u*v:=(u_1v_1,u_2v_2,\ldots, u_nv_n)^T$. Note if $A$ is a nonsingular matrix, then $A$ is a $P$-matrix if and only if $A$ is a $P_{\#}$-matrix. However, clearly a singular $P_{\#}$-matrix may not be a $P$-matrix, shown next.

\begin{ex}
Let $A \in \mathbb{R}^{n \times n}$ be the square matrix with all entries equal to $1$. If $x \in R(A)$, then one has $x=\alpha e$, where $e$ is the all ones vector, for some $\alpha \in \mathbb{R}$. This means that $Ax=n \alpha e$ and so, if $x_i(Ax)_i \leq 0$ for all $i=1,2,\ldots, n$, then one has $x=0$. This shows that $A$ is a $P_{\#}$-matrix. Such a matrix, being non-invertible, is not a $P$-matrix. 
\end{ex}

The example above is part of a subclass of $P_{\#}$-matrices, as we show in the following theorem.

\begin{thm}\label{genidemPhash}
Let $A$ be such that $A^2=\alpha A$, for some $\alpha >0$ (we may refer to such matrices as {\it generalized idempotent}). Then $A$ is a $P_{\#}$-matrix.
\end{thm}
\begin{proof}
Let $x \in R(A)$ so that $x=Ay$ for some $y$. Then $Ax=A^2y=\alpha Ay= \alpha x$. It now follows that if $x_i(Ax)_i \leq 0$, then $x_i(Ax)_i =\alpha x_i^2$ and hence $x=0$, showing that $A$ is a $P_{\#}$-matrix.
\end{proof}

\begin{rem}\label{genidem}
Since an idempotent matrix is generalized idempotent, it follows that for any $A \in \mathbb{R}^{m \times n}$, all the matrices $AA^{\dagger}, A^{\dagger}A, I-AA^{\dagger}$ and $I-A^{\dagger}A$ are $P_{\#}$-matrices. Also, if $A \in \mathbb{R}^{n \times n}$ is group invertible, then the matrices $AA^{\#}$ and $I-AA^{\#}$ are $P_{\#}$-matrices. So is the case for the Householder matrix $A=I-uu^T$, where $u \in \mathbb{R}^n$ satisfies the condition that $\parallel u \parallel =1$. More generally, if $A=I-uv^T$, where $u,v\in \mathbb{R}^n$ are such that $v^Tu=1$, then one may show that $A$ is a $P_{\#}$-matrix. Let us consider rank one matrices, separately in the next result. All the statements made here are in stark contrast to the case of $P$-matrices, due to the fact that the only generalized idempotent $P$-matrices are positive multiples of the identity matrix. 
\end{rem}

\begin{lem}\label{rankonephash}
Let $u,v(\neq 0)\in \mathbb{R}^n$ and let $A=uv^T$. Then $A$ is a $P_{\#}$-matrix if and only if $v^Tu > 0$. 
\end{lem}
\begin{proof}
The necessity part follows from the fact that $A$ is generalized idempotent. Conversely, suppose that a rank one matrix $A=uv^T$ is a $P_{\#}$-matrix. Note that $u \in R(A)$ and so $Au=v^Tu . u$ so that, if $v^Tu \leq 0$, then $u_i(Au)_i \leq 0$ for all $i$. This implies that $u=0$, a contradiction. Hence $v^Tu >0$. 
\end{proof}

\begin{rem}\label{prelimphash}
Before proceeding further, let us note that, analogous to the case of $P$-matrices, it is known that $A$ is a $P_{\#}$-matrix if, and only if, $A^{\#}$ (exists and) is a $P_{\#}$-matrix ({\cite[Theorem 5.1]{rajkcs}}  and {\cite[Theorem 2.11]{ijkcs}}). A proof is included for the sake of completeness. First we shall show that for a $P_{\#}$-matrix $A$, $A^{\#}$ exists. Let $x\in R(A) \cap N(A)$. Then $x_i(Ax)_i=0$ for all $i=1,2,...,n$, so that $R(A)\cap N(A)=\{0\}$. Since $R(A)$ and $N(A)$ are complementary, $A^{\#}$ exists. Now note that $R(A)=R(A^{\#})$. Let $y\in R(A^{\#})$ be such that $y_i(A^{\#}y)_i\leq 0$ for all $i=1,2,...,n$. Let $x=A^{\#}y$. Then $Ax=AA^{\#}y=y$, since $y\in R(A)$. Thus $x$ is such that $x\in R(A)$ and $x_i(Ax)_i\leq 0$ for all $i=1,2,...,n$. Since $A$ is a $P_{\#}$-matrix, $x=0$. Thus $y=Ax=0$ showing that $A^{\#}$ is a $P_{\#}$-matrix. We will use this result in the sequel. Also, if $A$ is a $P_{\#}$-matrix, then for every $q$, it follows that there exists at most one vector $x$ satisfying:  $x \geq 0, x \in R(A), y=Ax+q \geq 0$ and $x^Ty =0$ (\cite[Theorem 2.11]{ijkcs}).
\end{rem}

\begin{rem}\label{pnotphash}
Another extension of a $P$-matrix is the notion of a $P_0$-matrix, which is quite well-studied in the literature. Recall that $A$ is called a $P_0$-matrix, if all its principal minors are nonnegative. Apparently, there does not seem to be any nice relationship between $P_0$-matrices and $P_{\#}$-matrices. Consider the matrix $M_3$ in \cite{fl}: $A=\begin{pmatrix}
0 & -1 & -2 \\ 0 & ~~1 & ~~2\\ 1 & ~~1 & ~~1
\end{pmatrix}.$ Clearly, $A$ is not a $P_0$-matrix. However, $A$ is a $P_{\#}$-matrix: Let $x \in R(A)$ so that $x=(\alpha, -\alpha, \beta)^T$ for some $\alpha, \beta \in \mathbb{R}$. Then one has $Ax=(\alpha - 2\beta, -\alpha+2\beta, \beta)^T$. Thus, if $x_i(Ax)_i \leq 0$, for $i=1,2,3$, then $\alpha = \beta =0$, proving that $x=0$. One may generalize this example. Let $a \in \mathbb{R}^n$ be such that $a_1=0, ~a_2 < 0$ and $a_3=2a_2$, and let $E \in \mathbb{R}^{(n-2) \times n}, ~n \geq 3$ be the all ones matrix. Define $A=\begin{pmatrix}
~~a^T \\ -a^T\\ E
\end{pmatrix} \in \mathbb{R}^{n \times n}.$ Then the determinant of the $2\times 2$ principal submatrix obtained after deleting the first row, the first column and the last $n-3$ rows and columns of $A$ is $a_2 < 0$ and so $A$ is not a $P_0$-matrix, while it may be shown to be a $P_{\#}$-matrix. We omit the details. On the other hand, an example of a $P_0$-matrix which is not a $P_{\#}$-matrix is $B=\begin{pmatrix}
1 & 1 & 0 \\ 1 & 1 & 0\\ 0 & 1 & 0
\end{pmatrix}$ (matrix $M_1$ in \cite{fl}). Clearly, $B$ is a $P_0$-matrix. However, since the vector $x=(0,0,1)^T$ belongs to $N(B) \cap R(B)$, it follows that $B^{\#}$ does not exist and so $B$ is not a $P_{\#}$-matrix. 

The matrix $A$ above also serves to demonstrate that the transpose of a $P_{\#}$-matrix need not be a $P_{\#}$-matrix, unlike the case of $P_0$ or $P$-matrices. Let $C=\begin{pmatrix}
~~0 & 0 & 1 \\ -1 & 1 & 1 \\ -2 & 2 & 1
\end{pmatrix}$, which is the transpose of the $P_{\#}$-matrix $A$ above. If $x^0=(2,1,0)^T$, then $x^0 \in R(C)$ and $Cx^0=(0,-1,-2)^T$. Clearly, the sign of $x^0$ is reversed by $C$ and so it is not a $P_{\#}$-matrix. 
\end{rem}

\begin{rem}\label{aplusep}
Here is another observation. A well known result (Theorem 3.4.2, \cite{cps}) states that $A$ is a $P_0$-matrix if, and only if, $A+\epsilon I$ is a $P$-matrix for all $\epsilon >0$. It is interesting to observe that an analogous statement for $P_{\#}$-matrices does not hold. For the matrix $A$ of Remark \ref{pnotphash}, if for instance, $0 < \epsilon \leq \frac{1}{4},$ then the trailing $2 \times 2$ principal submatrix of $A+\epsilon I$ has a negative determinant. If $B$ is as given in Remark \ref{pnotphash}, since it is a $P_0$-matrix, one has that $B+\epsilon I$ is a $P$-matrix for all $\epsilon >0$.

Nevertheless, one has the following: Let $A$ be a $P_{\#}$-matrix. Then for every $\epsilon >0$, the matrix $A+\epsilon I$ is invertible. For, let $(A+\epsilon I)x=0$, so that $Ax=-\epsilon x$. Thus $x \in R(A)$ and $x_i(Ax)_i = -\epsilon {x_i}^2 \leq 0$ for each $i$. Since $A$ has $P_{\#}$-property, it follows that $x=0$, showing that $A+\epsilon I$ is invertible. Let us explore this idea a little further. 

Suppose that $A$ has the property that $-1$ is not an eigenvalue. Then the Cayley transform ${\cal C}(A)$ of $A$ is defined by ${\cal C}(A):=(I+A)^{-1}(I-A).$ It is shown (Theorem 3.1, \cite{ft}) that if $A$ is (even a complex) $P$-matrix, then $F={\cal C}(A)$ is well defined and that both the matrices $I+F$ and $I-F$ are $P$-matrices. In particular, if $A$ is a real $P$-matrix, then for all $\epsilon >0$, it follows (by replacing $A$ by the matrix $\frac{1}{\epsilon}A$) that $I+G_{\epsilon}$ and $I-G_{\epsilon}$ are $P$-matrices, where $G_{\epsilon}:= ({\epsilon}I+A)^{-1}({\epsilon}I-A).$ However, this statement is false for $P_{\#}$-matrices. Again, for the matrix $A$ of Remark \ref{pnotphash}, while $G_{\epsilon}$ is well defined, observe that, if $0 < \epsilon \leq \frac{1}{4},$ then one has $I+G_{\epsilon} = 2\epsilon (A+{\epsilon}I)^{-1}=\dfrac{2}{(1+\epsilon)^2}\begin{pmatrix}
\epsilon^2+2\epsilon -1 & 2 & -\epsilon -1 \\
\epsilon -1 & \epsilon^2+\epsilon +2 & -\epsilon -1 \\
2\epsilon & -2\epsilon & \epsilon ^2+\epsilon
\end{pmatrix}.$ This is not a $P_{\#}$-matrix, as it is not a $P$-matrix, due to the fact that (the diagonal entry)  $\epsilon^2+2\epsilon -1 < 0$ for $0 < \epsilon \leq \frac{1}{4}.$
\end{rem}

\begin{ex}\label{diag}
From the definition, it is clear that all the diagonal entries of a $P_0$-matrix are nonnegative. This however, is not true for a $P_{\#}$-matrix. Let $A=\begin{pmatrix}
~~2 & ~~1 \\ -2 & -1 \\ 
\end{pmatrix}$ so that ($A$ has a negative diagonal entry and) $A=uv^T$, with $u=(1,-1)^T$ and $v=(2,1)^T$. Since $v^Tu >0$, it follows that $A$ is a $P_{\#}$-matrix. It is interesting to observe that since $A$ is a generalized idempotent matrix which is not a $P_0$-matrix, it follows that the earlier discussion on such matrices being $P_{\#}$-matrices brings in a certain exclusivity for matrices that are endowed with such a property. 

The same example above serves to illustrate the fact that adding a diagonal matrix with positive diagonal entries, to a $P_{\#}$-matrix, does not result in a $P_{\#}$-matrix, whereas this property is known to be true for both $P_0$ as well as $P$-matrices. Note that for all $\epsilon >0$, the matrix $A+\epsilon I=\begin{pmatrix}
~~2+\epsilon & ~~1 \\ -2 & -1+\epsilon \\ 
\end{pmatrix}$ is invertible. However, for $\epsilon <1$, it is not a $P_{\#}$-matrix, since it is not a $P$-matrix. 

Let us also observe that while any principal submatrix of a $P_0$-matrix or a $P$-matrix inherits such a property, a principal submatrix of a $P_{\#}$-matrix need not be a $P_{\#}$-matrix (\cite[Remark 2.7]{ijkcs}).
\end{ex}

In what follows, we identify another rather distinguished class of $P_{\#}$-matrices. Before that, however, we need a certain perspective, which is provided next. Recently, a study was undertaken in \cite{ijkcs}, where the following result deriving necessary conditions for a $Z$-matrix to be a $P_{\#}$-matrix was proved. It is useful to observe that this extends certain items of Theorem \ref{zq} and brings in a connection to the linear complementarity problem (providing a statement stronger than what was mentioned earlier).  

\begin{thm}(\cite[Theorem 3.1]{ijkcs})\label{necphash}
Let $A$ be a $Z$-matrix. Consider the following statements:\\
$(a)~A$ is a $P_{\#}$-matrix.\\
$(b)~A$ is an $M$-matrix with ``property $c$."\\
$(c)~A$ is range monotone.\\
$(d)~A^{\#}$ exists and $A^{\#}x \geq 0$ whenever $x\in \mathbb{R}^n_+ \cap R(A)$.\\
$(e)~Ax\leq 0$ and $x\in \mathbb{R}^n_+ \cap R(A)\Longrightarrow x=0$.\\
Then $(a)\Longrightarrow (b) \Longleftrightarrow (c) \Longleftrightarrow (d) \Longrightarrow (e)$.
\end{thm}

\begin{rem}\label{singirrphash}
The question of whether a general $Z$-matrix is a $P_{\#}$-matrix, if it is an $M$-matrix with ``property $c$," (posed in \cite{ijkcs}) remains open. However, for $Z$-matrices of order $2 \times 2$ and $3 \times 3$, a proof was supplied to show that $M$-matrices with ``property $c$'' are $P_{\#}$-matrices. 

The authors of \cite{ijkcs} also showed that interestingly, for the class of symmetric $Z$-matrices the statements of Theorem \ref{necphash} are equivalent (Corollary 3.2, \cite{ijkcs}). Let us make use of this in providing a class of examples of matrices that satisfy the $P_{\#}$-property. Let $A=I-B$, where $B$ is an irreducible symmetric row-stochastic matrix (meaning that all the entries of $B$ are nonnegative and each row sum equals $1$). Then by item $(c)$ of Theorem \ref{singirr}, it follows that such a matrix is an $M$-matrix with ``property $c$''. By the result for symmetric matrices mentioned here, it now follows that $A$ is a $P_{\#}$-matrix. A specific numerical example is provided by the matrix $A=\frac{1}{3}\begin{pmatrix}
~~2 & -1 & -1\\ -1 & ~~2 & -1\\ -1 & -1 & ~~2
\end{pmatrix}.$ 

We would like to add another condition that has been shown to be equivalent to each of the statements of Theorem \ref{necphash}, for symmetric matrices, namely the strict range semimonotonicity condition (Corollary 3.2, \cite{ijkcs}). This will be useful in proving that a certain symmetric matrix satisfies the $P_{\#}$-property. Matrix $A$ is said to be strictly range semimonotone, if 
\begin{center}
$x \in R(A),~x \geq 0$ and $x * Ax \leq 0 \Longrightarrow x=0,$
\end{center}
where $*$ is the Hadamard entrywise product, defined earlier. 

Let us close this remark observing that there are singular reducible $M$-matrices that are not $P_{\#}$-matrices, reinforcing the fact (mentioned earlier) that symmetric singular irreducible $M$-matrices form a special class of $P_{\#}$-matrices. Let $A=\begin{pmatrix}
0 & -1 \\ 0 & ~~0
\end{pmatrix}$. Then $A^{\#}$ does not exist and so $A$ is not a $P_{\#}$-matrix. Note that if $B=\begin{pmatrix}
1 & 1 \\ 0 & 1
\end{pmatrix},$ then $B \geq 0, \rho(B)=1$ and $A=\rho(B)I-B$, so that $A$ is an $M$-matrix. 
\end{rem}

Next, we present another subclass of $P_{\#}$-matrices.

Ingleton, in \cite{ing}, introduced a class of matrices called adequate matrices. A matrix $A$ is said to be {\it adequate} if it satisfies the following two conditions:
\begin{enumerate}
\item All the principal minors of $A$ are nonnegative.
\item Each {\it vanishing} principal minor has the property that the associated rows and columns in the matrix $A$, are linearly dependent. 
\end{enumerate}

Clearly, every adequate matrix is a $P_0$-matrix. A characterization of nonsingular adequate matrices is then presented in \cite{ing}. An invertible matrix is adequate if and only if all its principal minors are positive, i.e., an invertible matrix $A$ is adequate if and only if $A$ is a $P$-matrix. Cottle \cite{cottle}, proved the following theorem, extending this result.

\begin{thm}\cite[Theorem 2]{cottle}\label{adequate sign reversal}
Let $A$ be an adequate matrix of order $n\times n$ and let $y=Ax$. If $x_iy_i\leq 0$ for $i=1,2,...,n$, then $y=0$.
\end{thm}

Using this result, we show that a certain class of $P_{\#}$matrices is contained in the class of adequate matrices.

\begin{thm}\label{adequatephash}
Let $A$ be a group invertible adequate matrix. Then $A$ is a $P_{\#}$-matrix.
\end{thm}
\begin{proof}
Let $x\in R(A)$ be such that $x_i(Ax)_i\leq 0$ for all $i=1,2,...,n$. Since $A$ is an adequate matrix, from Theorem \ref{adequate sign reversal}, we have $Ax=0$. Thus we have $x\in R(A)\cap N(A)$. But since $A$ is group invertible, $R(A)\cap N(A)=\{0\}$. Thus we have $x=0$ and hence $A$ is a $P_{\#}$-matrix.
\end{proof}

\begin{rem}
Let $A=\begin{pmatrix}
~~2 & ~~1 \\ -2 & -1 \\ 
\end{pmatrix},$ considered in Example \ref{diag}. It is shown there that this is a $P_{\#}$-matrix. $A$ however, is not an adequate matrix since it is not a $P_0$-matrix. This shows that the converse of Theorem \ref{adequatephash} is not true, showing that the class of adequate matrices is distinct from the class of $P_{\#}$-matrices.
\end{rem}

Next, we describe a process by which one may construct a $P_{\#}$-matrix whose leading principal minor is a symmetric $Z$-matrix which is also a $P$-matrix (i.e., a symmetric invertible $M$-matrix). We shall need the result that if $A$ is an irreducible invertible $M$-matrix, then $A^{-1} >0$ (meaning that all the entries of $A^{-1}$ are positive) \cite{berpl}.

\begin{pro}\label{algo2}
Let $A \in \mathbb{R}^{n \times n}$ be a symmetric irreducible invertible $M$-matrix. Define $B=\begin{pmatrix}
A & u \\u^T & \alpha,
\end{pmatrix}$
where $u$ is any nonzero nonpositive vector and $\alpha =u^TA^{-1}u$. Then $B$ is a $P_{\#}$-matrix.  
\end{pro}
\begin{proof}
First, note that, since $A^{-1} > 0$, it follows that $\alpha >0$. Also, $\alpha-u^TA^{-1}u$ is the Schur complement of $A$ in $B$, which is zero here. Hence the matrix $B$ is singular. Since $B$ is a $Z$-matrix, proving that $B$ is a $P_{\#}$-matrix will be achieved by showing that $B$ is strictly range semimonotone. The conclusion would then follow from the comments made in Remark \ref{singirrphash}. \\
To show the strict range semimonotonicity of $B$, we must show that $x \geq 0, ~x \in R(B)$ and $x * Bx \leq 0 \Longrightarrow x=0$. So, let $x \geq 0$ and $x \in R(B)$. Then $x=(w^T, \gamma)^T$, where $w=Az + \beta u$ and $\gamma = u^Tz + \alpha \beta$ for some vector $z$ and real $\beta.$ If we set $v=A^{-1}u \leq 0,$ (so that $u = Av$), it follows that $\gamma =v^Tw$.  Since $x \geq 0$, one has $w \geq 0$ and $\gamma \geq 0$. Thus $0 \geq v^Tw = \gamma \geq 0$ and so $\gamma =0$. We have $$x*Bx = ((w, 0)* ((Aw)^T, u^Tw))^T$$ and so $0 \geq x*Bx \Longrightarrow 0 \geq w*Aw$. Since $A$ is a $P$-matrix, we have $w=0$, showing that $B$ is a $P_{\#}$-matrix.  
\end{proof}

\begin{ex}
Here is a numerical illustration of the method above: Let $A=\begin{pmatrix}
~~1 & -1 \\
-1 & ~~2
\end{pmatrix}$ and let $u=-(1,1)^T$. Let $\alpha = u^TA^{-1}u=5$. If $B=\begin{pmatrix}
~~1 & -1 & -1\\
-1 & ~~2 & -1\\
-1 & -1 & ~~5
\end{pmatrix},$ constructed by the procedure above, then $B$ is a $P_{\#}$-matrix. Here is a direct verification, that turns out to be quite simple for the present example. If $x \in R(B)$, then $x=(\beta,\gamma, -3\beta-2\gamma)^T$, for scalars $\beta, \gamma$. The requirement that $x \geq 0$ immediately implies that $x=0$, proving that $B$ is a $P_{\#}$-matrix.
\end{ex}

\begin{rem}\label{algo2rem}
The procedure described in Proposition \ref{algo2}, seems applicable to matrices that are not necessarily symmetric, too. A general proof for the non-symmetric case appears elusive. Let $A=\begin{pmatrix}
~~1 & -1 \\ -2 & ~~3
\end{pmatrix}$, then $A=3I-C$, where $C=\begin{pmatrix}
2 & 1 \\ 2 & 0
\end{pmatrix}$, so that $\rho(C)=1+\sqrt{3}<3$. Thus $A$ is a non-symmetric irreducible invertible $M$-matrix. Let $u=-(1,1)^T$. Let $\alpha =u^TA^{-1}u=7$. If $B=\begin{pmatrix}
\ \ 1 &-1 & -1 \\
-2 & \ \ 3 & -1 \\
-1 & -1 & \ \ 7
\end{pmatrix}$, then $R(B)=span\{(a, b, -(5a+2b))^T: a,b\in \mathbb{R}\}$. If $x\in R(B)$, then $x=(\gamma, \delta, -(5\gamma +2\delta) )^T$ for $\gamma, \delta \in \mathbb{R}$. Then $Bx=(6\gamma 
+\delta , 3\gamma+5\delta, -36\gamma -15\delta)^T$. $x_i(Bx)_i\leq 0$ for all $i$ gives $\gamma(6\gamma +\delta)\leq 0$, $\delta(3\gamma +5\delta)\leq 0$ and $(5\gamma +2\delta )(36\gamma +15\delta)\leq 0$. It is easy to see that $\gamma =0$ if, and only if, $\delta =0$. If $\gamma ,\delta $ are both positive then from the first inequality, we have $6\gamma \leq -\delta <0$. Similarly, if $\gamma, \delta$ are both negative, then $6\gamma \geq -\delta >0$. We get a contradiction in both the cases. Also if $\gamma >0$ and $\delta <0$, we have from the first two inequalities, $3 \gamma \geq -5\delta \geq 30 \gamma$, which is not possible since $\gamma >0$. If $\gamma <0$ and $\delta >0$, we have $3\gamma \leq -5\delta \leq 30 \gamma$, which is again not possible since $\gamma <0$. Hence $\gamma=\delta =0$ and so $x=0$. Thus $B$ is a $P_{\#}$-matrix. \\

The procedure seems to be applicable for reducible matrices, too. Let $A=\begin{pmatrix}
1 & -1 \\ 0 & ~~2
\end{pmatrix}$. Then $A$ is a reducible non-symmetric invertible $M$-matrix. Let $u=(-1,-1)^T$ so that $\alpha =2$. It is easy to check that $B=\begin{pmatrix}
~~1 & -1 & -1 \\ ~~0 & ~~2 & -1 \\ -1 & -1 & ~~2
\end{pmatrix}$ is a $P_{\#}$-matrix.
\end{rem}

We conclude this section with two more examples of $P_{\#}$-matrices, which also will be useful in later discussions.

\begin{ex}\label{phashex1}
Let $A=\begin{pmatrix}
1 & 1 & 1  \\
0 & 1 & 1  \\
0 & 0 & 0
\end{pmatrix}$. Then $A^{\#} = \begin{pmatrix}
1 & -1 & -1 \\
0 & \ \ 1 & ~~1 \\
0 & \ \ 0 & ~~0
\end{pmatrix}.$ We shall show that $A$ and $A^{\#}$ are both $P_{\#}$-matrices. But by Remark \ref{prelimphash}, showing $A$ is $P_{\#}$ is sufficient to show that $A^{\#}$ is $P_{\#}$. For $x\in R(A)$, $x=(\alpha, \beta, 0)^T$ for some $\alpha, \beta \in \mathbb{R}$. Then $Ax=(\alpha +\beta, \beta, 0)^T$. So $x_i(Ax)_i\leq 0$ for all $i$ gives $\alpha =\beta =0$ and hence $x=0$. Thus $A$ is a $P_{\#}$-matrix and hence $A^{\#}$ is also a $P_{\#}$-matrix.
\end{ex}

\section{Karamardian Matrices}\label{karamardian}

\subsection{Definition and Preliminary Observations}

We begin by proposing the notion of Karamardian matrices. As was mentioned earlier, the motivation comes from Theorem \ref{karthm}, which gives a sufficient condition for a matrix to be a $Q$-matrix. We need the following notation: For a given matrix $A \in \mathbb{R}^{n \times n}$, a nonempty set $K \subseteq \mathbb{R}^n_+$, and for a given $q \in \mathbb{R}^n$, the problem LCP$(A,K,q)$ is to determine if there exists $x \in \mathbb{R}^n$ such that $x \in K,~y:=Ax+q  \in K^*$ with $x^Ty =0$. Here, $K^*$, called the {\it dual of K} is defined by 
\begin{center}
$K^*:=\{y\in \mathbb{R}^n: x^Ty \geq 0$ for all $x \in K\}.$
\end{center}

Note that the usual complementarity problem is denoted using two arguments, viz., a matrix $A$ and a vector $q$ whereas, the problem above makes use of three arguments, a matrix $A$, the subset $K$ and a vector $q$. We shall be interested in a specific choice of $K$ as described below.

\begin{defn}\label{karamdef}
Let $A \in \mathbb{R}^{n \times n}$. Then $A$ is said to be a Karamardian matrix, if it satisfies the following two conditions:\\
$(a)$ $K=\mathbb{R}^n_+ \cap R(A) \neq \{0 \}$.\\
$(b)$ The problems LCP$(A,K,0)$ and LCP$(A,K,d)$ for some $ d\in \rm int(K^*)$, have a unique solution, namely zero.
\end{defn}

If $A$ is a Karamardian matrix, we may sometime say that $A$ is Karamardian. Let us paraphrase the definition above. $A$ is said to be Karamardian if $K=\mathbb{R}^n_+ \cap R(A) \neq \{0 \}$ and for some $d \in \rm int(K^*)$, the problems LCP$(A,K,td)$ have zero as the only solution for $t=0,1$. Observe that zero is always a solution to these two problems. We may sometimes refer to LCP$(A,K,0)$ as the {\it homogeneous problem} and to LCP$(A,K,d)$ as the {\it non-homogeneous problem.}

\begin{rem}
For $K=\mathbb{R}^n_+ \cap R(A)$, it is well known that $K^* = \mathbb{R}^n_+ + N(A^T)$. So, if $x \in K,$ and $y=Ax+q \in K^*,$ then one has $y=u+v$, where $u \in \mathbb{R}^n_+$ and $v \in N(A^T)$. Thus, Definition \ref{karamdef} is set in the frame work of a cone complementarity problem. 
\end{rem}

\begin{rem}
It is known that if $L$ and $M$ are convex subsets of a topological vector space with $\rm int(L)\neq \{0\}$, then $\rm int(L+M)= \rm int(L) + M$ \cite[Theorem 2.2]{tankur}. As $K^* = \mathbb{R}^n_+ + N(A^T)$, it now follows that $\rm int (K^*)=\rm int(\mathbb{R}^n_+) + N(A^T).$ Thus, if $d\in \rm int(K^*),$ then one may write $d=a+b$, where $a> 0$ and $b\in N(A^T)$. We shall use this notation for $d$ henceforth. Also, note that for $x\in R(A)$, $x^Td=x^Ta$ since $x^Tb=0$.
\end{rem}

\begin{rem}
Suppose that there exists $x>0$ such that $A^Tx=0$. Assume that $K\neq \{0\}$. Let $0 \neq y \geq 0$ be such that $y\in R(A)$. Then $\langle x,y \rangle =0$. However, since $x>0$ and $y\geq 0$, one has $\langle x,y \rangle >0$, a contradiction. So, if $K\neq \{0\}$, then $N(A^T)$ does not contain any positive vector (and so does not contain any negative vector). Due to this, it follows that if $K \neq \{0\}$, then $d\neq 0$. This shows that the problems LCP$(A,K,0)$ and LCP$(A,K,d)$ for $d \in \rm int(K^*)$ are distinct problems, under the assumption that $K \neq \{0\}$.
\end{rem}

\begin{rem}\label{rcapn}
Suppose that for the matrix $A$, there exists a nonzero nonnegative vector $x \in R(A) \cap N(A).$ Then, one has $x \in R(A), x\geq 0$ and $Ax=0$ so that $x^TAx=0,$ showing that the homogeneous problem has a nonzero solution. Hence, $A$ is not a Karamardian matrix.
\end{rem}

\begin{rem} A Karamardian matrix is  defined in terms of two cone complementarity problems possessing a unique solution and at this point in time, we are not aware of any application of this notion, in the theory of the cone complementarity problem. Nevertheless, the results in the subsection \ref{monotonicity and kar} demonstrate that this class of matrices has interesting positivity properties, especially for the subclass of $Z$-matrices. This, we believe is an important motivation  for their consideration.
\end{rem}



Let $A$ be invertible. If $A$ is a Karamardian matrix, then $A$ is a $Q$-matrix. This is a consequence of Theorem \ref{karthm}. We shall show that the converse is not true. That is, we show that an invertible $Q$-matrix is not necessarily a Karamardian matrix; in fact we identify a class of invertible $Q$-matrices that do not satisfy the Karamardian theorem (Theorem \ref{karthm}) and hence are not Karamardian matrices (Theorem \ref{ms} below; a numerical example is also included in Example \ref{Qnotkar}). Thus, the class of $Q$-matrices is distinct from the class of Karamardian matrices, even for the subclass of invertible matrices.

\begin{thm}\cite[Theorem 4.2]{ms92}\label{ms}
Let $A\in \mathbb{R}^{n\times n}$ have at least one positive entry in each column. Then, $A$ is an $N$-matrix of the first category if and only if LCP$(A,q)$ has a unique solution for all $q \ngeq 0$, exactly three solutions for all $q > 0$, and at most two solutions for any other $q \in \mathbb{R}^n_+$. In particular, if $A$ is an $N$-matrix of the first category, then $A$ is a $Q$-matrix.
\end{thm}

\begin{ex}\label{Qnotkar}
Let $A=\begin{pmatrix}
-1 & -2 & \ \ 1 \\
-1 & -1 & \ \ 3 \\
\ \ 2 & \ \ 1 & -1
\end{pmatrix}$. Then $A$ is an $N$-matrix of the first category. By the theorem as above, $A$ is a $Q$-matrix and it is not a Karamardian matrix.
\end{ex}

A useful fact about Karamardian matrices is that they are preserved under permutation similarities, as is shown next.

\begin{thm}\label{permutation similarity thm}
Suppose $A$ is a Karamardian matrix. Then $PAP^T$ is a Karamardian matrix, where $P$ is any permutation matrix. 
\end{thm}
\begin{proof}
Suppose $A$ is Karamardian. Then $K = \mathbb{R}^n_{+} \cap R(A) \neq \{0\}$, LCP$(A,K,0)$ has only the trivial solution, and there exists a $ d \in \rm int(K^*)$ such that LCP$(A,K,d)$ has only the trivial solution.

Let $P$ be any permutation matrix. Since $\mathbb{R}^n_{+} \cap R(A) \neq \{0\}$, there exists an $x \in \mathbb{R}^n$ such that $Ax = y$ where $0 \neq y \geq 0$. Set $z = Px$. Then $PAP^T z = PAx = Py$ is a nonzero nonnegative vector in $R(PAP^T)$. Thus, $K_1:=\mathbb{R}^n_{+} \cap R(PAP^T) \neq \{0\}$.

First, assume that LCP$(PAP^T,K_1,0)$ has a solution. Then there is $x \in R(PAP^T)$ such that $x \geq 0, ~PAP^Tx =r+w, \ r\geq 0,\ w\in N(PA^TP^T)$ and $x^T PAP^Tx = 0$. Set $y = P^T x$ so that $y \geq 0$. Observe that $Ay =P^{-1}r+P^{-1}w$. Since $P$ is a permutation matrix, one has $P^{-1}r \geq 0$. Also, $PA^TP^Tw=0$ implies that $A^TP^{-1}w=0$ so that $P^{-1}w \in N(A^T)$. As $x \in R(PAP^T)$, there exists a $z$ such that $PAP^Tz = x$ so that $AP^Tz = P^Tx = y$ which implies that $y \in R(A)$. Further, $y^T A y = 0$. We have shown that $y \geq 0, ~Ay \in \mathbb{R}^n_+ +N(A^T)$ and $y^T A y = 0$. Since LCP$(A,K,0)$ has only the trivial solution, one has $y=0$ and so $x=0$, showing that LCP$(PAP^T,K_1,0)$ has only the trivial solution.

Since $A$ is Karamardian, there exists $d \in \rm int(K^*)$, such that LCP$(A,K,d)$ has zero as the only solution. First, observe that $K_1^*=\mathbb{R}^n_+ + N(A^TP^T)$. Now, consider $Pd$. Writing $d=a+b$ with $a > 0$ and $A^Tb=0$, one has $Pd=g+h$, where $g=Pa > 0$ and $h=Pb$. Since $A^TP^Th=A^TP^TPb=A^Tb=0$, it follows that $Pd \in \rm int(K_1^*)$. Let $x$ be a solution of LCP$(PAP^T,K_1,Pd)$; we claim that $x=0$. Then $x\geq 0$, $x\in R(PAP^T)$, $PAP^Tx+Pd\in K_1^*$ and $x^T(PAP^Tx+Pd)=0$. Now $P^Tx\geq 0$ and $P^Tx\in R(A)$. Write $PAP^Tx+Pd=q+r$, where $q \geq 0$ and $A^TP^Tr=0$. Premultiplying by $P^T$, one obtains $AP^Tx+d=P^Tq+P^Tr=u+v$, where, $u=P^Tq \geq 0$, while $A^Tv=A^T(P^Tr)=0$. This means that $A(P^Tx)+d\in K^*$. Finally, 
$(P^Tx)^T(AP^Tx+d)=x^T(PAP^Tx+Pd)=0$. Thus $P^Tx$ is a solution of LCP$(A,K,d)$, which however, has zero as the only solution. Thus $x=0,$ as was claimed. 
\end{proof}

\subsection{Karamardian Matrices and Range Monotonicity}\label{monotonicity and kar}

In this subsection, we derive some relationships between Karamardian matrices and range monotone matrices.

Let us recall some results that involve extensions of the idea of nonnegativity of the inverse of an invertible matrix. Let $A$ be a real square matrix. Then, $A$ is said to be {\it monotone} if $$Ax \geq 0 \Longrightarrow x\geq 0.$$ It is easy to show that $A$ is monotone if and only if $A$ is invertible and $A^{-1} \geq 0$. This means that the following implication holds: $$x \geq 0 \Longrightarrow A^{-1}x\geq0.$$ If a rectangular matrix $A$ satisfies the monotonicity condition above, then one has the characterization that $A$ has a nonnegative left inverse. This means that there exists $Y \geq 0$ such that $YA=I$. This implies that one has: $$YA=I {~\textit{and}}~x \geq 0 \Longrightarrow Yx \geq0.$$
A square matrix $A$ is said to be {\it group monotone} if $$Ax \in \mathbb{R}^n_+ +N(A), ~x \in R(A) \Longrightarrow x\geq 0.$$
A necessary and sufficient condition for a matrix to be group monotone is that $A^{\#}$ exists and that it is nonnegative. This means that
$$x \geq 0 \Longrightarrow A^{\#}x \geq 0.$$ A weaker notion than group monotonicity is recalled next. Square matrix $A$ is said to be {\it range monotone} if 
$$Ax \geq 0, ~x \in R(A) \Longrightarrow x \geq 0.$$ Then, $A$ is range monotone if and only if $A^{\#}$ exists and one has the implication $$x \geq 0~ {\textit{and}}~x \in R(A) \Longrightarrow A^{\#}x \geq 0.$$ A rectangular matrix $A$ is said to be {\it semimonotone} (a nomenclature not to be confused with a matrix class with a similar name, in LCP theory) if $$Ax \in \mathbb{R}^n_+ +N(A^T), ~x \in R(A^T) \Longrightarrow x\geq 0.$$ A necessary and sufficient condition for a matrix to be semimonotone is that $A^{\dag}$ is nonnegative. Needless to say that this means that $$x \geq 0 \Longrightarrow A^{\dag}x \geq 0.$$ Finally, we recall a notion that is entirely similar to range monotonicity. 
$A$ is said to be {\it row monotone} if 
$$Ax \geq 0, ~x \in R(A^T) \Longrightarrow x \geq 0.$$ Then, $A$ is row monotone if and only if one has the implication $$x \geq 0~ {\textit{and}}~x \in R(A^T) \Longrightarrow A^{\dag}x \geq 0.$$ It may be shown that row monotonicity of a rectangular matrix is weaker than semimonotonicity. For more details, we refer the reader to the survey \cite{bpeight}.

Here is a brief overview of the main results of this section. First, we obtain a result for rank one matrices, which is of independent interest. As consequences, we show that a rank one Karamardian matrix and its group inverse are range monotone (Theorem \ref{rankonerangemonotone}). It is also shown that a rank one Karamardian matrix has the property that the Moore-Penrose inverse and the transpose of the matrix, are row monotone (Theorem \ref{rankonerowmonotone}). In Theorem \ref{ZKrangemonotone}, we show that all $3 \times 3$ upper triangular $Z$-matrices with nonnegative diagonal entries, which are also Karamardian, must be range monotone. Theorem \ref{rmonkar} presents a condition for a $Z$-matrix $A$ in order for $A^{\#}$ to be a Karamardian matrix.

Recall that a non-zero vector $x\in \mathbb{R}^n$ is said to be unisigned if $x\geq 0$ or $x\leq 0$.

\begin{thm}\label{rankonekar}
Let $A=uv^T$, for nonzero $u,v\in \mathbb{R}^n$. Then $A$ is Karamardian if, and only if, $u$ is unisigned and $u^Tv >0$.
\begin{proof}
\noindent Necessity: Suppose that $A=uv^T$ is Karamardian. Since $K=span\{u\} \cap \mathbb{R}^n_+$, we have that $u$ is unisigned. If $u^Tv =0$, then $Au=0$. Choose $x=u$, if $u \geq 0$ and $x=-u$, if  $u \leq 0$, ensuring that $x \in R(A)$ and that $x \geq 0$. As $Ax=0$, it follows that $x$ is a nonzero solution of LCP$(A,K,0)$, contradicting that $A$ is Karamardian. Hence $u^Tv \neq 0$. 

Suppose that $u^Tv<0$. Let $d$ be arbitrarily chosen such that $d\in \rm int(K^*)$. Then $d=a+b$, where $a > 0$ and $b\in N(A^T)$. Irrespective of the sign of $u$, if we set $x=-\frac{u^Ta}{u^Tv \Vert u \Vert ^2}u$, then $x \geq 0,$ and $x \in R(A)$. If $x=0,$ then $u^Ta=0$, which is impossible since $a >0,$ as this would mean that $u=0$. Thus $x$ is a nonzero solution for LCP$(A,K,d)$ which, however contradicts the fact that $A$ is Karamardian. Hence $u^Tv>0$. \\
\noindent Sufficiency: Suppose $u$ is unisigned and $u^Tv >0$. Then $K \neq \{0\}$. Let $x$ be a nonzero solution of LCP$(A,K,0)$. Then $x=\beta u \geq 0$, $Ax=\beta u^Tv u \in K^*$ and $x^TAx=\beta ^2  u^T v  \Vert u \Vert^2 =0$. Since $x\neq 0$, $\beta \neq 0$. So, $u=0$, a contradiction. Thus $x=0$, showing that the homogeneous problem has only the trivial solution. \\
Assume, without loss of generality that $u \geq 0$. Let $\epsilon >0$ be chosen arbitrarily. Define the vector $\epsilon_u$ by $(\epsilon_u)_i=\epsilon$ if $u_i=0$ and $(\epsilon_u)_i=0$ otherwise. Then $u^T\epsilon_u=0$. Choose $d=u+\epsilon_u$. (If $u \leq 0,$ then one may choose $d=-u+\epsilon_u$). Clearly, $d>0$ and so $d \in \rm int(K^*)$. Let $y$ be a nonzero solution of LCP$(A,K,d)$. Then $y=\beta u\geq 0$, $Ay+d=(\beta  u^Tv )u+u+\epsilon_u \in \mathbb{R}^n_+ \subseteq K^*$ and $y^T(Ay+d)=\beta(\beta u^Tv +1) \Vert u \Vert ^2+\beta u^T\epsilon_u=0$. Again since $y\neq 0$, we get $u=0$, a contradiction. Thus $y=0$ and hence $A$ is Karamardian.
\end{proof}
\end{thm}

In what follows, we present some consequences of Theorem \ref{rankonekar}.

It is known that for an invertible matrix $A$, one has: $A$ is a $Q$-matrix if and only if $A^{-1}$ is a $Q$-matrix. An example is presented later to show that the inverse of an invertible Karamardian matrix need not be Karamardian, in general (Remark \ref{invertibleK}). Nevertheless, in the next result, in particular, we show that an analogous result holds for Karamardian matrices of rank one, when one employs the group inverse. We also obtain an interesting monotonicity type result (see also Theorem \ref{rankonerowmonotone}).

\begin{thm}\label{rankonerangemonotone}
Let $A$ be a matrix of rank one. Then $A$ is Karamardian if and only if $A^{\#}$ is a Karamardian matrix. Further, in this case, both $A$ and $A^{\#}$ are range monotone.
\end{thm}
\begin{proof}
Let $A=uv^T$. By Theorem \ref{rankonekar}, $u$ is unisigned and $u^Tv >0$. One may verify that $A^{\#}=\frac{1}{{(u^Tv)}^2}A,$ a positive multiple of $A$. It is now clear that $A^{\#}$ is Karamardian. The converse part follows from the identity $(A^{\#})^{\#}=A$. \\
Next, we show that $A$ is range monotone. Suppose that $u \geq 0.$ Let $x \in R(A)$ be such that $Ax \geq 0$. Then, $x=\alpha u$, for some $\alpha \in \mathbb{R}$ and so $0 \leq Ax=(\alpha u^Tv)u$. We then have $\alpha \geq 0$ and so $x \geq 0$. Thus, $A$ is range monotone. Clearly, a similar conclusion holds, if $u \leq 0$. Again, since $A^{\#}$ is a positive multiple of $A$, it follows that $A^{\#}$ is also range monotone.
\end{proof}

The next result is a version of Theorem \ref{rankonerangemonotone} for the Moore-Penrose inverse.

\begin{thm}\label{rankonerowmonotone}
Let $A$ be a rank one matrix. Then, $A$ is a Karamardian matrix if and only if $(A^{\dag})^T$ is Karamardian. Further, in such a case, $A^T$ as well as $A^{\dag}$ are row monotone.
\end{thm}
\begin{proof}
It is easy to observe that $A^{\dag}$ is a positive multiple of $A^T$ and so $(A^{\dag})^T$ is a positive multiple of $A$. Also $(A^{\dag})^{\dag}=A$. Hence, the first part follows. \\
Next, let $A=uv^T$ be a Karamardian matrix and $B=A^T$. We must show that $Bx \geq 0, x\in R(B) \Longrightarrow x\geq 0.$ Let $x \in R(B)$ so that $x=\alpha v$. If $Bx \geq 0$ then, $0 \leq Bx=\alpha A^Tv=(\alpha u^Tv) v=(u^Tv) x$ so that $x \geq 0$. Thus, $A^T$ is row monotone. The proof for the row monotonicity of $A^{\dag}$ is entirely similar.
\end{proof}

In what follows, we further explore the relationship between Karamardian matrices and range monotone matrices. In particular, we consider the following question: If a $Z$-matrix is Karamardian, then is it range monotone? We show that the answer to this question is negative (Remark \ref{nonutZKnotrangemono}). We know that a $Z$-matrix which is also a $Q$-matrix is invertible and is a monotone matrix. Also such a matrix is a $P$-matrix and hence has positive diagonal entries. We shall now show an analogous result holds for a subclass of Karamardian matrices of order $3 \times 3$. We believe that this result should be true for general $n \times n$ matrices.

\begin{thm}\label{ZKrangemonotone}
Let $A\in \mathbb{R}^{3\times 3}$ be an upper triangular $Z$-matrix with nonnegative diagonal entries. If $A$ is Karamardian, then $A$ is range monotone.
\end{thm}
\begin{proof}
Let us recall that if $A$ is a rank one Karamardian matrix, then $A$ is range monotone. The cases that we consider here exclude matrices of rank one. Since we are looking for singular matrices, at least one of the diagonal entries must be zero. First, consider the case when all the diagonal entries are zero. Then the matrix is of the form $A=\begin{pmatrix}
0 & - & \ominus \\
0 & 0 & \ominus \\
0 & 0 & 0
\end{pmatrix}$. In this case, $(1,0,0)^T\in R(A)\cap N(A)$ and hence the homogeneous problem has a non-zero solution. Thus, $A$ is not Karamardian.\\

Next, let precisely two diagonal entries of $A$ be zero. Then $A$ will assume one of the following three forms: $A_1=\begin{pmatrix}
1 & \ominus & \ominus \\
0 & 0 & \ominus \\
0 & 0 & 0
\end{pmatrix}$, 
$A_2=\begin{pmatrix}
0 & \ominus & \ominus \\
0 & 1 & \ominus \\
0 & 0 & 0
\end{pmatrix}$ and 
$A_3=\begin{pmatrix}
0 & \ominus & \ominus \\
0 & 0 & \ominus \\
0 & 0 & 1
\end{pmatrix}$. Then, one has $(1,0,0)^T\in R(A_i)\cap N(A_i)$, for $i=2,3$ and $(0,1,0)^T \in R(A_1)\cap N(A_1)$. Thus, by Remark \ref{rcapn}, it follows that none of the three matrix classes listed above, belongs to the class of Karamardian matrices.

Finally, consider the case when precisely one of the diagonal entries is zero. We then have the following classes of matrices: $A_1=\begin{pmatrix}
0 & - & \ominus \\
0 & 1 & \ominus \\
0 & 0 & +
\end{pmatrix}$ and $A_2=\begin{pmatrix}
0 & \ominus & - \\
0 & 1 & - \\
0 & 0 & +
\end{pmatrix}$. In either case, $K=\{0\}$. Hence $A_1,A_2$ are not Karamardian matrices. The rest of the cases, where the matrix has exactly one diagonal entry equal to zero need to be investigated. We claim that the matrices in such cases are all Karamardian matrices.
It follows that, there are just four types of matrices as given below, that need to be considered. \\
(1) $A=\begin{pmatrix}
0 & 0 & \alpha \\
0 & 1 & \beta \\
0 & 0 & \gamma
\end{pmatrix}$, where $\alpha, \beta \leq 0$ and $\gamma >0$.\\
(2) $A=\begin{pmatrix}
0 &  \alpha & 0 \\
0 & 1 & 0 \\
0 & 0 & \beta
\end{pmatrix}$, where $\alpha \leq 0$ and $\beta >0$. \\
(3) $A=\begin{pmatrix}
1 & \alpha & \beta \\
0 & 0 & \gamma \\
0 & 0 & \delta
\end{pmatrix}$, where $\alpha, \beta, \gamma \leq 0$ and $\delta >0$.\\
(4) $A=\begin{pmatrix}
1 & \alpha & \beta \\
0 & \gamma & \delta \\
0 & 0 & 0
\end{pmatrix}$, where $\alpha, \beta, \delta \leq 0$ and $\gamma >0$.\\ 
We prove the result by considering each of these matrices.\\
Let $A=\begin{pmatrix}
0 & 0 & \alpha \\
0 & 1 & \beta \\
0 & 0 & \gamma
\end{pmatrix}$, where $\alpha, \beta \leq 0$ and $\gamma >0$. Then $K\neq \{0\}$ since the second column is nonnegative. Let $x\in R(A)$. Then $x$ is of the form $x=(\alpha y_3, y_2+\beta y_3, \gamma y_3)^T$, so that  $Ax=(\alpha \gamma y_3, y_2+\beta y_3+\beta \gamma y_3, {\gamma}^2y_3)^T$. Then, $x^TAx=0$ gives $y_3=0$, which in turn, implies that $y_2=0$. Thus $x=0$ is the only solution of LCP$(A,K,0)$. In an entirely similar manner, it follows that LCP$(A,K,e)$ has zero as the only solution. Thus $A$ is a $Z$-matrix, which is also Karamardian. Now we shall show that $A$ is range monotone. Let $x \in R(A)$ and $Ax \geq 0.$ If $\alpha \neq 0$, then from the condition that the first and the last coordinates of $Ax$ are nonnegative, one obtains that $y_3=0$, which in turn, implies that $y_2 \geq 0$. Thus, $x \geq 0$. If $\alpha =0$, then $x=(0,y_2+\beta y_3, \gamma y_3)^T,$ so that $Ax=(0, y_2+\beta y_3 + \beta \gamma y_3, {\gamma}^2y_3)^T$. Again, $x^TAx=0$ implies that $y_3=0,$ which yields that $y_2=0$. Thus $x=0$.
Similarly, we can show that the following matrices are Karamardian matrices and can also be shown to be range monotone.

Next, we consider (2). Let $A=\begin{pmatrix}
0 &  \alpha & 0 \\
0 & 1 & 0 \\
0 & 0 & \beta
\end{pmatrix}$, where $\alpha \leq 0$ and $\beta >0$. Since the third column is nonnegative, $K\neq \{0\}$. Let $x \in R(A)$. Then $x=(\alpha y_2, y_2, \beta y_3)^T,$ so that $Ax=(\alpha y_2, y_2, \beta^2 y_3)^T$. The condition $x^TAx=0$ yields $x=0$. Thus, zero is the only solution of LCP$(A,K,0)$. Similarly, we can show that $x=0$ is the only solution of LCP$(A,K,e)$. Thus $A$ is a Karamardian matrix.

Let $x\in R(A)$ so that $x=(\alpha y_2, y_2, \beta y_3)^T$ and $Ax =( \alpha y_2, y_2, \beta^2 y_3)^T$. $Ax \geq 0$ at once yields $y_3 \geq 0$ and $y_2 \geq 0$ and so $x\geq 0$. This shows that $A$ is range monotone.

Next, let us consider matrices of type (3). Here, $A=\begin{pmatrix}
1 & \alpha & \beta \\
0 & 0 & \gamma \\
0 & 0 & \delta
\end{pmatrix}$, where $\alpha, \beta, \gamma \leq 0$ and $\delta >0$. Since the first column of $A$ is nonnegative, $K\neq \{0\}$. Let $x\in R(A)$. Then $x=(y_1+ \alpha y_2 +\beta y_3, \gamma y_3, \delta y_3)^T$, so that $Ax=(y_1+ \alpha y_2 +\beta y_3+\alpha \gamma y_3+\beta \delta y_3, \gamma \delta y_3, \delta^2 y_3)^T$. The condition $x^TAx=0$ yields $y_3=0$ and $y_1+\alpha y_2=0$ and thus $x=0$ is the only solution for LCP$(A,K,0)$. Similarly, we can show that $x=0$ is the only solution for LCP$(A,K,e)$. Thus $A$ is a Karamardian matrix.

Let $x\in R(A)$ so that $x=(y_1+ \alpha y_2 +\beta y_3, \gamma y_3, \delta y_3)^T$ and $Ax=(y_1+ \alpha y_2 +\beta y_3+\alpha \gamma y_3+\beta \delta y_3, \gamma \delta y_3, \delta^2 y_3)^T$. $Ax\geq 0$ yields $y_3\geq 0$. If $\gamma =0$, then we have $Ax=(y_1+\alpha y_2+\beta y_3+\beta \delta y_3, 0, \delta ^2y_3)^T\geq 0$. Since $\delta >0$, $\beta \leq 0$ and $y_3\geq 0$, $\beta \delta y_3\leq 0$ and hence $0\leq (Ax)_1-\beta \delta y_3 =x_1$. Thus we have $x\geq 0$. If $\gamma <0$, then from $Ax\geq 0$, we have $y_3=0$. This implies $y_1+\alpha y_2\geq 0$ which in turn implies $x\geq 0$. Thus $A$ is range monotone.

Finally, let us turn to matrices of the form given in (4). So, let $A=\begin{pmatrix}
1 & \alpha & \beta \\
0 & \gamma & \delta \\
0 & 0 & 0
\end{pmatrix}$, where $\alpha, \beta, \delta \leq 0$ and $\gamma >0$. Since the first column of $A$ is nonnegative, $K\neq \{0\}$. Let $x\in R(A)$. Then $x=(y_1+\alpha y_2 +\beta y_3, \gamma y_2+\delta y_3, 0)^T$. $Ax=(y_1+\alpha y_2 +\beta y_3+\alpha \gamma y_2 +\alpha \delta y_3, \gamma ^2y_2+\gamma \delta y_3, 0)^T$. $x_2(Ax)_2=0$ yields $\gamma(\gamma y_2+\delta y_3)^2=0$. Since $\gamma >0$, we have $\gamma y_2+\delta y_3 =0$. Thus from $x_1(Ax)_1=0$, we get $x=0$. $x=0$ is the only solution of LCP$(A,K,0)$. Similarly, we can show that LCP$(A,K,e)$ has zero as the only solution.

Let $x\in R(A)$ be such that $Ax\geq 0$. Since $\gamma>0$, $(Ax)_2\geq 0$ implies $x_2\geq 0$. Also, since $\alpha \leq 0$ and we have $x_2=\gamma y_2+\delta y_3\geq 0$, $0\leq (Ax)_1-\alpha x_2=x_1$. Thus $x\geq 0$ showing that $A$ is range monotone.\\
This completes the proof.
\end{proof}

\begin{rem}
The condition that $A$ is Karamardian cannot be dropped in Theorem \ref{ZKrangemonotone}. In other words, an upper triangular  $Z$-matrix whose diagonal entries are nonnegative need not be range monotone. Let $A=\begin{pmatrix}
1 & -1 & -1 \\
0 & \ \ 0 & -1 \\
0 & \ \ 0 & \ \ 0
\end{pmatrix}$. $A$ is not Karamardian as $(1,1,0)^T$ is a nonnegative vector in $R(A)\cap N(A)$. Further, if $x=(1,-1,0)^T$, then $x \in R(A)$ is a vector for which $x \ngeq 0$ while $Ax=(2,0,0)^T\geq 0$. Thus $A$ is not range monotone.
\end{rem}

\begin{rem}\label{nonutZKnotrangemono}
We show that one cannot hope to get a result similar to Theorem \ref{ZKrangemonotone}, if the restriction that the matrix is upper triangular, is removed. Consider the symmetric tridiagonal $Z$-matrix $A=\begin{pmatrix}
\ \ 0 & -\beta & \ \ 0 \\
-\beta & \ \ 0 & -\beta \\
\ \ 0 & -\beta & \ \ 0 
\end{pmatrix}$, $\beta >0$. Observe that $R(A)$ is spanned by the set $\{(0,1,0)^T, (1,0,1)^T\}$ and that $N(A^T)=N(A)=span\{(1,0,-1)^T\}$. Let $x=(\alpha, \gamma, \alpha)^T\geq 0$ be a solution of LCP$(A,K,0)$, so that, in particular, one has $x \geq 0$ and $Ax \geq 0.$ Then $\alpha, \gamma \geq 0$ and since $Ax=(-\beta \gamma, -2 \beta \alpha, -\beta \gamma)^T\geq 0$, one concludes at once that $x=0$. Thus zero is the only solution for LCP$(A,K,0)$. Now let $d=(3,1,-1)^T$. Then, $d=a+b$, where $a=(1,1,1)^T >0$ and $(2,0,-2)^T \in N(A^T)$. Thus, $d \in \rm int(K^*)$. Let $y=(\delta, \eta, \delta)^T\geq 0$ be a solution of LCP$(A,K,d)$ so that $Ay+d=(-\beta \eta +3, -2 \beta \delta+1, -\beta \eta-1)^T\geq 0$. If $\delta \neq 0$, then by the complementarity condition, the first and the third coordinates of $Ay+d$ should be zero, leading to an inconsistency. Hence $\delta =0$. This means that the second coordinate of $Ay+d$ is nonzero so that the second coordinate of $y$ must be zero implying that $\eta = 0$. This proves that $y=0$ is the only solution for LCP$(A,K,d)$. This shows that $A$ is Karamardian. Also, if $x\in R(A)$ satisfies $Ax\geq 0$, then $x\leq 0$ and so, $A$ is not range monotone.
\end{rem}

\begin{rem}\label{nonutkarrangm}
Let $A=\begin{pmatrix}
~~1 & -1 & 0\\ -1 & ~~1 & 0 \\~~0 & ~~0 & 1
\end{pmatrix}$ be as given in Example \ref{stcopex}. It is shown there that $A$ is a Karamardian matrix. It may be verified that $A$ is range monotone. This is a non-upper triangular Karamardian $Z$-matrix, which is range monotone.
\end{rem}

In view of Remark \ref{nonutZKnotrangemono} and Remark \ref{nonutkarrangm}, one may ask: which type of symmetric, tridiagonal $Z$-matrices with nonnegative diagonal entries, that satisfy the Karamardian property, are also range monotone.

Next, we turn around and ask which type of $Z$-matrices possess Karamardian property.

Let us recall that if a $Z$-matrix is monotone, then it is a $Q$-matrix. Also, an invertible $Q$-matrix $A$ has the property that $A^{-1}$ is a $Q$-matrix. Thus the inverse of a monotone $Z$-matrix is a $Q$-matrix. We have the following analogue for Karamardian matrices. 

\begin{thm}\label{rmonkar}
Let $A$ be a range monotone $Z$-matrix with $K\neq \{0\}$. Then $A^{\#}$ is Karamardian.
\begin{proof}
Let $x$ be a solution of LCP$(A^{\#},K,0)$. Then $x\in K$, $A^{\#}x \in K^*$ and $x^TA^{\#}x =0$. Since $A$ is range monotone, $x\in R(A)$ and $x\geq 0$, one has $A^{\#}x\geq 0$. Since $A$ is a $Z$-matrix, this implies that $0 \geq x^TAA^{\#}x = x^Tx =0$ and so $x=0$. Thus LCP$(A,K,0)$ has a unique solution. Let $u$ be a solution of LCP$(A^{\#},K,d)$ for some $d=a+b\in \rm int(K^*)$. Then $u\in K$, $v=A^{\#}u+d \in K^*$ and $u^Tv =0$. Since $A$ is range monotone, $u\in R(A), \ u\geq 0$ implies $A^{\#}u\geq 0$. Now, 
\begin{center}
$0=u^Tv = u^T(A^{\#}u+d) = u^TA^{\#}u + u^Td=u^TA^{\#}u + u^Ta$. 
\end{center}
But, since $u,A^{\#}u$ and $a$ are nonnegative vectors, one has, in particular, $u^TA^{\#}u=0$. Finally, since $u \geq 0, u^TA^{\#}u \geq 0$ and $A$ is a $Z$-matrix (using $u=AA^{\#}u,$ since $u \in R(A)$) one has $0\geq u^TAA^{\#}u = u^Tu$, showing that $u=0$.
\end{proof}
\end{thm}

Let us give an example of a matrix illustrating Theorem \ref{rmonkar}. A distinguished class of matrices for which Theorem \ref{rmonkar} is applicable, is identified in Corollary \ref{rmonkarcor}. 

\begin{ex}\label{zandrangemon}
Let $A=\begin{pmatrix}
1 & -1 & ~~0 \\ 0 & ~~1 & -1 \\ 0 & ~~0 & ~~0
\end{pmatrix}$. Then $A$ is a $Z$-matrix. Also, $K=R(A) \cap \mathbb{R}^n_+ \neq \{0\}$, since $A$ has a nonnegative column, for instance. It is easy to see that $A$ is range monotone. It now follows from Theorem \ref{rmonkar} that $A^{\#}$ is Karamardian. Again, it is easy to see that $A$ is not a $Q$-matrix.
\end{ex}

\begin{cor}\label{rmonkarcor}
Let $B$ be an $M$-matrix with ``property $c$," $C$ be an invertible $M$-matrix and let $A=B\oplus C=\begin{pmatrix}
B & 0 \\ 0 & C
\end{pmatrix}$. Then $A^{\#}$ is Karamardian. 
\begin{proof}
First, we show that $K\neq \{0\}$. Since $C$ is an invertible $M$-matrix, there exists $z >0$ such that $Cz >0$. Thus the vector $\begin{pmatrix} 0\\Cz \end{pmatrix} \in K$. Note that, as $B,C$ are both $Z$-matrices, $A$ is also a $Z$-matrix. Next, we show that $A$ is range monotone, i.e., $x\in R(A)$, $Ax\geq 0$ implies $x\geq 0$. Let $x=\begin{pmatrix}
u\\v
\end{pmatrix} \in R(A)$ such that $Ax=\begin{pmatrix}
B & 0 \\ 0 & C
\end{pmatrix}\begin{pmatrix}
u\\v
\end{pmatrix}=\begin{pmatrix}
Bu\\Cv
\end{pmatrix} \geq 0$, so that $Bu\geq 0$ and $Cv\geq 0$. Note that $u\in R(B)$ and $v\in R(C)$. Showing $u\geq 0$ and $v\geq 0$ will prove the range monotonicity of $A$. Since $B$ is an $M$-matrix with ``property $c$," we have that $B^{\#}$ exists and $B^{\#}\geq 0$ on $R(B)$. Thus, since $Bu\geq 0$, $B^{\#}Bu\geq 0$. Note that since $u \in R(B)$, one has $u=B^{\#}Bu\geq 0$. Since $C$ is an invertible $M$-matrix, we have $C^{-1}\geq 0$. Thus $Cv\geq 0$ implies that $v\geq 0$, proving that $A$ is range monotone. By Theorem \ref{rmonkar}, $A^{\#}$ is Karamardian.
\end{proof}
\end{cor}

In view of Corollary \ref{rmonkarcor}, one might think of Karamardian matrices, within the class of $Z$-matrices, as being sandwiched between the classes of invertible $M$-matrices and $M$-matrices with ``property $c$.'' Here is an illustrative example for Corollary \ref{rmonkarcor}. 

\begin{ex}
Let $B=\begin{pmatrix}
0 & -1 \\ 0 & \ \ 1
\end{pmatrix}$ and $C=\begin{pmatrix}
\ \ 1 & -1 \\ -1 & \ \ 2
\end{pmatrix}$. $C$ is an invertible $M$-matrix. Note that $B=I-D$, where $D=\begin{pmatrix}
1 & 1 \\ 0 & 0
\end{pmatrix}\geq 0$.
It follows that $D^k=D \geq 0$ for any positive integer $k$ so that $D$ is semi-convergent and that $\rho(D)=1$. Thus, $B$ is an $M$-matrix with ``property $c$.'' One has
$A^{\#}=\begin{pmatrix}
0 & -1 & 0 & 0 \\
0 & \ \ 1 & 0 & 0 \\
0 & \ \ 0 & 2 & 1 \\
0 & \ \ 0 & 1 & 1 \\
\end{pmatrix}$. Let $0 \leq x \in R(A^{\#})$. Then $x_1=x_2=0$. The condition that $x^TA^{\#}x =0$ yields $2{x_3}^2+2x_3x_4+{x_4}^2=0$. Since $x_3,x_4\geq 0$, we have $x_3=x_4=0$. Thus $x=0$ is the only solution of LCP$(A^{\#},K,0)$. Similarly, it can be shown that LCP$(A^{\#},K,d)$, where $d=(1,1,1,1)^T\in \rm int(K^*)$, has zero as its only solution. Hence $A^{\#}$ is Karamardian. Since the leading principal submatrix of $A^{\#}$ is a $Z$-matrix which is not an $M$-matrix, it follows that $A^{\#}$ is not a $Q$-matrix.
\end{ex}

\begin{ex}
In this example we show that the conclusion in Corollary \ref{rmonkarcor} does not hold if both $B$ and $C$ are $M$-matrices, even if both possess ``property $c$." Let $B=C=\begin{pmatrix}
0 & -1 \\ 0 & \ \ 1
\end{pmatrix}$. As shown earlier, $B=C$ is an $M$-matrix with ``property $c$.'' One has $A^{\#}=A$ and $K=\{0\}$.
\end{ex}

At this point, it is pertinent to bring forth the main connection between the class of $P_{\#}$-matrices and the class of Karamardian matrices. From Theorem \ref{necphash}, we have that every $Z$-matrix which is a $P_{\#}$-matrix, is range monotone. Thus we have the following important consequence to Theorem \ref{rmonkar}. Two other results are presented in Corollary \ref{symmphashzkar} and Remark \ref{phashkar question}. These results vindicate the consideration of $P_{\#}$-matrices earlier.

\begin{cor}\label{phashzkar}
Let $A$ be a $P_{\#}$-matrix which is also a $Z$-matrix. Let $K\neq \{0\}$. Then $A^{\#}$ is Karamardian.
\end{cor}

In the next result, we show how one can strengthen the conclusion of Theorem \ref{rmonkar}, in the presence of an additional assumption on $A$. This result also improves Theorem 2.18 of \cite{kcslaa}. 

\begin{thm}\label{symmrange}
Let $A$ be a symmetric and range monotone $Z$-matrix. Let $A$ be such that $A(int(\rm \mathbb{R}^n_+)) \cap \mathbb{R}^n_+ \neq \emptyset$. Then $A$ and $A^{\#}$ are Karamardian.
\end{thm}
\begin{proof}
Observe that $A(\rm int(\mathbb{R}^n_+)) \cap \mathbb{R}^n_+ \neq \emptyset$ if and only if there exists $x>0$ such that $Ax\geq 0$. So, $K\neq \{0\}$. The fact that $A^{\#}$ is a Karamardian matrix is the result of Theorem \ref{rmonkar} (even without the symmetry assumption). We show that $A$ is Karamardian.
Now, set $d=x$ so that $d\in \rm int(K^*)$. We show that the homogeneous problem and the problem corresponding to the vector $d$, have zero as the only solution in $R(A)$. Let $u_t \in R(A)$ satisfy the conditions 
\begin{center}
$u_t \geq 0, v_t=Au_t + td =r_t+w_t,\ r_t\geq 0,\ w_t\in N(A^T)$ and $u_t^Tv_t =u_t^Tr_t=0,$  
\end{center}
for $t=0,1$. Since $A$ is a $Z$-matrix, we have 
\begin{center}
$0 \geq (Au_t)^T r_t = (Au_t)^T(Au_t+td-w_t)= {\parallel Au_t \parallel}^2 + t (Au_t)^Td.$
\end{center}
Note that $Au_t\in R(A)$ and $w_t\in N(A^T)$ and hence $(Au_t)^Tw_t=0$. Now,
\begin{center}
$(Au_t)^T d =u_t^T Ad=u_t^T Ax,$
\end{center}
since $A$ is symmetric, $u_t, Ax$ are nonnegative vectors. Thus $Au_t=0$ so that $u_t \in R(A) \cap N(A)=\{0\}$, completing the proof.
\end{proof}

Using the fact that a $Z$-matrix which is also a $P_{\#}$-matrix is a range monotone matrix, we have the following consequence.

\begin{cor}\label{symmphashzkar}
Let $A$ be a symmetric $P_{\#}$-matrix which is a $Z$-matrix. Suppose that $A(\rm int(\mathbb{R}^n_+)) \cap \mathbb{R}^n_+ \neq \emptyset$. Then $A$ and $A^{\#}$ are Karamardian.
\end{cor}

Here is an example that illustrates Theorem \ref{symmrange} (as well as Corollary \ref{symmphashzkar}).

\begin{ex}\label{rangemoneg}
Let $A=\begin{pmatrix}
~~1 & -1 & 0 \\ -1 & ~~1 & 0 \\ ~~0 & ~~0 & 3
\end{pmatrix}$. Then $A$ is a symmetric $Z$-matrix. Note that $Ae\geq 0$, where $e>0$ is the vector with all entries equal to 1. Also if $x\in R(A)$, then $x=(\alpha,-\alpha,\beta)^T$ for some $\alpha,\beta \in \mathbb{R}$. So, if $x\in R(A)$ and $Ax \geq 0$, then $(2\alpha,-2\alpha, 3\beta )^T\geq 0$ so that $x\geq 0$. Thus $A$ is range monotone. Also since $A$ has a nonnegative column, $K\neq \{0\}$. Hence by Theorem \ref{rmonkar}, $A^{\#}=\dfrac{1}{4}\begin{pmatrix}
~~1 & -1 & 0 \\ -1 & ~~1 & 0 \\ ~~0 & ~~0 & \dfrac{4}{3}
\end{pmatrix}$ is a Karamardian matrix. Since $A^{\#}$ also satisfies the conditions of Theorem \ref{rmonkar}, $(A^{\#})^{\#}=A$ is Karamardian. It is clear that $A$ is not a $Q$-matrix.
\end{ex}

\begin{rem}\label{phashkar question}
From the rank one characterizations presented in Theorem \ref{rankonephash} and Theorem \ref{rankonekar}, we observe that a rank one $P_{\#}$-matrix with $K\neq \{0\}$ is a Karamardian matrix. Also, a generalized idempotent matrix with $\alpha >0$ (which was shown earlier to be a $P_{\#}$-matrix) is Karamardian. A proof is included for the sake of completeness. Let $A^2=\alpha A$ for some $\alpha >0$. Let $x$ be a solution of LCP$(A,K,0)$. Then $x\in R(A)$, $x\geq 0$, $Ax=\alpha x \geq 0$ and $x^TAx=\alpha {\parallel x \parallel}^2 =0$. Since $\alpha >0$, $x=0$. Thus $x=0$ is the only solution of LCP$(A,K,0)$. Similarly, let $ d=a+b\in \rm int(K^*)$. Let $y$ be a solution of LCP$(A,K,d)$. Then $y\in R(A)$, $y\geq 0$, $Ay+d=\alpha y+d \in K^*$ and $0=y^T(Ay+d)=\alpha \Vert y \Vert ^2 +y^Td =\alpha \Vert y \Vert ^2 +y^Ta$. Since both the terms are nonnegative, we have $\alpha \Vert y \Vert ^2 =0$. Thus $y=0$. Hence $A$ is Karamardian.

In the above, we have identified two classes of $P_{\#}$-matrices that are also Karamardian matrices. An important question arises. {\it If $A$ is a $P_{\#}$-matrix such that $K\neq \{0\}$, is $A$ Karamardian?} At the moment, this question remains open.
\end{rem}

\subsection{Relationship With Other Matrix Classes}\label{other matrix classes}
Let us reiterate that our endeavour is to prove that Karamardian matrices possess many properties that may be considered as analogous to those that are satisfied by the class of $Q$-matrices. It is useful to keep the following result for $Q$-matrices in mind: Any $P$-matrix is a $Q$-matrix and that the inverse of an invertible $Q$-matrix is also a $Q$-matrix. We shall obtain a verbatim analogue of this statement for Karamardian matrices, which states that any $P$-matrix is Karamardian. 

The relationship of Karamardian matrices with other matrix classes is presented in this section. Here is an overview of some of the main results in this section. First $P$-matrices are shown to be Karamardian in Theorem \ref{pkar}. Nonsingular strictly semimonotone matrices are shown to be Karamardian in Theorem \ref{nonsingular strictly semi is Kar}. Theorem \ref{nonnegnonsingkar} gives a characterization for a nonnegative nonsingular matrix to be Karamardian. A class of strictly copositive matrices is shown to be Karamardian in Theorem \ref{karcop}. A bordering result for Karamardian matrices is given in Proposition \ref{karconstruct}.

Let us begin this subsection with the following remark.

\begin{rem}\label{positivekar}
It is quite well known that a positive matrix is a $Q$-matrix. One of the ways of proving this statement is by using Theorem \ref{karthm}. It is easily shown that a positive matrix is Karamardian, too. Let us include a proof, for the sake of completeness. Let $A$ be a positive matrix. Then $K=R(A)\cap \mathbb{R}^n_+ \neq \{0\}$. If $x\neq 0$, then $Ax>0$ and so $x^TAx\neq 0$. This means that $x=0$ is the only solution to LCP$(A,K,0)$. Next, for $d\in \rm int(K^*)$, if $x\neq 0$ is a solution to LCP$(A,K,d)$, then $x^T(Ax+d)=x^TAx+x^T(a+b)$, where $a> 0$ and $b\in N(A^T)$. Since $x\geq 0$ and $x\in R(A)$, $x^Ta > 0$ and $x^Tb=0$. Thus we have $x^T(Ax+d)>0$, a contradiction. Thus $x=0$ is the only solution to LCP$(A,K,d)$ for any nonzero $d\in K$. Thus $A$ is a Karamardian matrix.
\end{rem}

\begin{thm}\label{pkar}
Let $A$ be a $P$-matrix. Then $A$ and $A^{-1}$ are Karamardian matrices.
\end{thm}
\begin{proof}
Since $A$ is a $P$-matrix, by \cite[Corollary 3.3.5]{cps}, it follows that there exists $x >0$ such that $Ax >0$ and so $K \neq \{0\}$. Now, since $A$ is a $P$-matrix, it is invertible and so LCP$(A,q)$ has a unique solution for all $q\geq 0$. Hence LCP$(A,0)$ and LCP$(A,d)$ for some nonzero $d\in \rm int(K^*)$ have zero as the only solution. Thus $A$ is Karamardian.
\end{proof}

\begin{rem}\label{subclasses of P}
A matrix is said to be {\it positive definite} if $x^TAx>0$ for all nonzero $x\in \mathbb{R}^n$. A generalization of the class of positive definite matrices is defined next. A matrix $A$ is said to be {\it stable} if there exists a symmetric positive definite matrix $H$ such that $HA$ is positive definite. A matrix $A$ is said to be {\it diagonally stable} if there exists a positive diagonal matrix $D$ such that $DA$ is positive definite. Diagonally stable matrices have a significant role in LCP theory because they form a subclass of the class of $P$-matrices.\\
A matrix $A$ is said to be an {\it $H$-matrix} if there exists a vector $d >0$ such that for all $i=1,2,...,n$, $\vert a_{ii}\vert d_i > \displaystyle {\sum_{j\neq i} \vert a_{ij}\vert d_j}$. This class of matrices is a generalization of the class of diagonally dominant matrices. Any $H$-matrix with positive diagonal entries is proved to be diagonally stable \cite{cps}. Thus, an $H$-matrix with positive diagonal entries is a $P$-matrix. This notion of $H$-matrices is quite useful in localization of eigenvalues.\\
A matrix $A$ is said to be a mime if $A=(s_1I-B_1)(s_2I-B_2)^{-1}$, where $B_1,B_2\geq 0$, $s_1>\rho(B_1)$, $s_2>\rho(B_2)$ and there exists a vector $x\geq 0$ such that $s_1x>B_1x$ and $s_2x>B_2x$. The class of mimes includes $M$-matrices and inverse $M$-matrices in its fold. It has been proved that this class forms a subclass of $P$-matrices, \cite{mimes}. \\
Clearly, by Corollary \ref{pkar}, all these classes of matrices are Karamardian.
\end{rem}

\begin{rem}\label{invertibleK} We have from Theorem \ref{pkar} that if $A$ is a $P$-matrix, then both $A$ and $A^{-1}$ are Karamardian. But in general, the inverse of an invertible Karamardian matrix is not a Karamardian matrix. Consider $A=\begin{pmatrix}
1 & 2 \\
1 & 1
\end{pmatrix}$. Since $A$ is a positive matrix, by Remark \ref{positivekar}, $A$ is a Karamardian matrix. Now, $A^{-1}=\begin{pmatrix}
-1 & \ \ 2 \\
\ \ 1 & -1 
\end{pmatrix}$. Clearly $A^{-1}$ is an $N$-matrix of the first category and hence by Theorem \ref{ms}, LCP$(A,q)$ has exactly three solutions when $q>0$ and hence cannot be Karamardian.
\end{rem}

\begin{rem}
The converse of Theorem \ref{pkar} is not true. Consider the matrix $\begin{pmatrix}
\ \ 0 & 1 \\
-1 & 1
\end{pmatrix}$. The diagonal entries of a $P$-matrix must be positive. Since this is not satisfied, $A$ is not a $P$-matrix. But by Theorem \ref{one diag entry zero Kar thm}, $A$ is a Karamardian matrix.
\end{rem}

\begin{rem}\label{qzkar}
It was noted in Remark \ref{Qnotkar} that a $Q$-matrix need not be Karamardian. As mentioned earlier, a $Q$-matrix which is also a $Z$-matrix is a $P$-matrix and hence by Theorem \ref{pkar}, we have that every $Q$-matrix which is also a $Z$-matrix is Karamardian.
\end{rem}

Next, we make some connections between Karamardian matrices and nonsingular (strictly) semimonotone matrices. For some recent results we refer to the work \cite{TsatMegan}. A \textit{(strictly) semimonotone matrix} is a matrix $A \in \mathbb{R}^{n \times n}$ such that for every $x \in \mathbb{R}^n$ such that $0 \neq x \geq 0$ there exists an index $k$ such that $x_k > 0$ and $(Ax)_k \geq 0$ (resp. $(Ax)_k > 0$).  A well-known result in this regard is the statement: A  matrix $A$ is (strictly) semimonotone if and only if LCP$(A,q)$ has a unique solution for all $q > 0$ $(q \geq 0)$ (see \cite{cps}). 

Semimonotone and strictly semimonotone matrices have another characterization in terms of semipositive and weakly semipositive matrices. Recall that a matrix $A$ is \textit{semipositive}, denoted by $A \in S$, if there exists an $x > 0$ such that $Ax > 0$. A matrix $A$ is called \textit{weakly semipositive}, denoted by $A \in S_0$, if there exists a $0 \neq x \geq 0$ such that $Ax \geq 0$. For more details on semipositive and weakly semipositive matrices, see \cite{berpl, Fiedler, Johnson}. It has been shown (see \cite{cps}) that a matrix $A$ is (strictly) semimonotone if and only if $A$ and all its principal submatrices are weakly semipositive (semipositive). For some recent results on  (strictly) semimonotone matrices, we refer to \cite{TsatMegan}.

\begin{thm}\label{nonsingular strictly semi is Kar}
Every nonsingular strictly semimonotone matrix is Karamardian. 
\end{thm}
\begin{proof}
Suppose $A$ is a strictly semimonotone matrix. Then $A$ is semipositive and hence must have a positive vector in its column space. Also, note that LCP$(A,q)$ has a unique solution for every $q \geq 0$. Since $A$ is invertible, LCP$(A,K,0)$ has a unique solution and LCP$(A,K,d)$ has a unique solution for every $d \in \rm int(K^*)$. Thus, $A$ is Karamardian. 
\end{proof}

Note that any positive matrix is strictly semimonotone and this allows us to easily get the following result as already observed in Remark \ref{positivekar}.

\begin{cor} \label{positive matrix is Kar}
If $A \in \mathbb{R}^{n \times n}$ is a positive matrix, then $A$ is Karamardian.
\end{cor}
\begin{rem} It may or may not be the case that a semimonotone matrix (either singular or nonsingular) is Karamardian. For example, the matrices $A = \begin{pmatrix*}[r]
0 & -1 \\ 
0 & 1
\end{pmatrix*}$ and $B = \begin{pmatrix}
0 & 1 \\
0 & 1
\end{pmatrix}$ are both singular semimonotone matrices. It can be shown that $A$ is not Karamardian while $B$ is Karamardian.  Also, the matrices $C = \begin{pmatrix}
0 & 1 \\
1 & 0
\end{pmatrix}$ and $D = \begin{pmatrix*}[r]
0 & 1 \\
-1 & 1
\end{pmatrix*}$ are both nonsingular semimonotone matrices (which are not strictly semimonotone). It can be shown that $C$ is not Karamardian while $D$ is Karamardian. 
\end{rem}

If $A$ is a semimonotone matrix which is not strictly semimonotone, then LCP$(A,q)$ does not have a unique solution for some $0 \neq q \geq 0$. If this $q \neq 0$, then $A$ may be Karamardian. If $q = 0$, then $A$ also may or may not be Karamardian since we are not guaranteed that the nontrivial solution is in the column space of $A$.
In the case of nonsingular semimonotone matrices, we have the following result.

\begin{thm}\label{semi and Kar thm}
Suppose $A$ is semimonotone but not strictly semimonotone and suppose that $A$ is nonsingular. If LCP$(A,0)$ has only the trivial solution, then $A$ is Karamardian.
\end{thm}
\begin{proof}
Suppose $A$ is semimonotone but not strictly semimonotone, and suppose that $A$ is nonsingular and that LCP$(A,0)$ has only the trivial solution. Since $A$ is semimonotone, LCP$(A,q)$ has only the unique solution for every $q > 0$. Thus, since each of these vectors is in $\rm int(K^*)$, $A$ is Karamardian. 
\end{proof}

Let us now prove a result regarding \textit{almost} semimonotone matrices. A matrix $A$ is said to be \textit{almost semimonotone} if all proper principal submatrices of $A$ are semimonotone but $A$ is not semimonotone. It has been shown (see \cite{TsatMegan}) that if $A$ is almost semimonotone, then $A^{-1}$ exists and $A^{-1} \leq 0$, which allows us to prove the following result.
\begin{thm}\label{almost semi is not Kar}
If $A \in \mathbb{R}^{n \times n}$ is almost semimonotone, then $A$ is not Karamardian.
\end{thm}
\begin{proof}
Suppose $A \in \mathbb{R}^{n \times n}$ is almost semimonotone. Then $A^{-1}$ exists and $A^{-1} \leq 0$. Since an invertible Karamardian matrix is a $Q$-matrix and since a nonpositive matrix is not a $Q$-matrix, $A$ is not Karamardian.
\end{proof}

%
%
%

As has been mentioned earlier, a nonnegative matrix is a $Q$-matrix if, and only if, all its diagonal entries are positive. We have a verbatim analogue for invertible Karamardian matrices.

\begin{thm}\label{nonnegkar}
Let $A$ be a nonnegative and nonsingular matrix. If $A$ has at least one zero diagonal entry, then $A$ is not Karamardian.
\end{thm}
\begin{proof}
Let $A \in \mathbb{R}^{n \times n}$. Without loss of generality, assume that $a_{11} = 0$. This is possible due to Theorem \ref{permutation similarity thm}. Then $A$ can be written in the form
\begin{align*}
A = \begin{pmatrix}
0 & u^T \\
v & B
\end{pmatrix}
\end{align*}
where $u, v \in \mathbb{R}^{n-1}$ and $B \in \mathbb{R}^{(n-1) \times (n-1)}$. If $x = (1,0,0, \ldots, 0)$, then $0 \neq x \geq 0$, $Ax = (0, v) \geq 0$, and $x^T A x = 0$, showing that LCP$(A,0)$ has a nontrivial solution. Thus, $A$ is not Karamardian.
\end{proof}

\begin{cor}\label{nonneg nonsing kar}
Suppose $A \geq 0$ and nonsingular. If $A$ is Karamardian, then all the diagonal entries of $A$ are positive. 
\end{cor}

We shall now show that the converse of Corollary \ref{nonneg nonsing kar} is true for singular matrices, too.

\begin{thm}\label{nonneg to be kar}
Suppose $A\geq 0$. If all the diagonal entries of $A$ are positive, then $A$ is Karamardian.
\end{thm}
\begin{proof}
Clearly, $K=R(A)\cap \mathbb{R}^n_+\neq \{0\}$. Let $x$ be a nonzero solution of LCP$(A,K,0)$. Let $x_k>0$. Note that $Ax\geq 0$ and since $A$ has positive diagonal entries, we have $(Ax)_k>0$. Thus $x^TAx=0$ is violated, a contradiction. Hence $x=0$ is the only solution for LCP$(A,K,0)$. By a similar argument, we can show that $x=0$ is the only solution for LCP$(A,K,d)$ for any $d\in \rm int(K^*)$.
\end{proof}

The converse of the above result is not true, as shown below.

\begin{ex}
Consider the matrix \begin{align*}
A = \begin{pmatrix}
0 & 0 & 1 & 0 \\
0 & 0 & 1 & 1 \\
0 & 0 & 0 & 1 \\
0 & 0 & 1 & 0
\end{pmatrix}
\end{align*}
Note that $A$ is singular and $A \geq 0$ {\it has all its diagonal entries zero}. We claim that $A$ is Karamardian. First, since $A\geq 0,$ it is trivial to observe that $R(A)$ has a nonnegative basis. Thus, $K\neq \{0\}$. Next, we show that LCP$(A,K,0)$ has only the trivial solution. Let $x \in K$. Then $x$ is in the form
\begin{align*}
x = \begin{pmatrix}
\alpha \\ 
\alpha + \beta \\
\beta \\
\alpha
\end{pmatrix},
\end{align*}
for some real numbers $\alpha, \beta \geq 0$. Note that  $Ax = \begin{pmatrix}
\beta \\
\beta + \alpha \\
\alpha \\
\beta
\end{pmatrix}$ and so, $x^T Ax = 3\alpha \beta + (\alpha + \beta)^2$. Thus, $x^T A x = 0$ if and only if $\alpha = \beta = 0$. Thus, the only solution to LCP$(A,K,0)$ is the trivial one.  Next, let $d = \begin{pmatrix}
1 \\ 4 \\ 3 \\ 1
\end{pmatrix}$ so that $d \in \rm int(K^*)$. 
Next, let $x \in R(A)$. Then, using the form of $x$, as before, one has, 
\begin{align*}
Ax + d = \begin{pmatrix}
\beta + 1 \\
\beta + \alpha + 4 \\
\alpha + 3 \\
\beta + 1
\end{pmatrix}.
\end{align*}
If one imposes the condition $x \geq 0$, then $Ax+d>0$.  Hence, $x^T(Ax + d) = 0$ if and only if $x = 0$, showing that the problem LCP$(A,K,d)$ has zero as the only solution. This completes the proof that $A$ is a Karamardian matrix.  
\end{ex}
Combining the above results, we have the following characterization for nonnegative nonsingular Karamardian matrices.

\begin{thm}\label{nonnegnonsingkar}
Let $A\geq 0$ be nonsingular. Then $A$ is Karamardian if and only if $A$ is a $Q$-matrix.
\end{thm}

Before we proceed to identify other classes of Karamardian matrices, let us give two examples of matrices that are not Karamardian.

\begin{ex}\label{singirrnotkar}
Let $A$ be a singular irreducible $M$-matrix. Let $x \geq 0$ with $x \in R(A)$. Then $0 \leq x=Ay$ and so by item $(e)$ of Theorem \ref{singirr}, one has $x=Ay=0$. This shows that $K=\{0\}$ and so $A$ is not a Karamardian matrix. It is pertinent to point that, nevertheless, there is a procedure to construct a Karamardian matrix via a direct sum of a singular $M$-matrix and an invertible $M$-matrix (see Corollary \ref{rmonkarcor}).
\end{ex}

\begin{ex}
Let $A=\begin{pmatrix}
E & E \\ E & 0
\end{pmatrix} \in \mathbb{R}^{2n \times 2n}$, where $E \in \mathbb{R}^{n \times n}$ is the all ones matrix. $A$ is not Karamardian since $\begin{pmatrix}
0 \\ e
\end{pmatrix}$ is a nonzero solution of LCP$(A,K,0)$. Here $e \in\mathbb{R}^n$ denotes the vector all of whose entries are equal to 1.
\end{ex}

Let us recall that a matrix $A$ is said to be {\it copositive (strictly copositive) on a subset $S \subseteq \mathbb{R}^n$}, if $x^TAx \geq 0$ ($x^TAx >0$) for all $x \in S$ ($0 \neq x \in S$). $A$ is said to be {\it copositive (strictly copositive)}, if $S=\mathbb{R}^n_+$. One of the first results for a $Q$-matrix is that a strictly copositive matrix is a $Q$-matrix (\cite[Theorem 3.8.5]{cps}). The next result is an analogue for a Karamardian matrix. 

\begin{thm}\label{karcop}
Let $K \neq \{0\}$. If $A$ is strictly copositive on $K$, then $A$ is Karamardian.
\end{thm}
\begin{proof}
Let $x \geq 0, x \in R(A), y=Ax \in K^*$ and $x^Ty =0$. If $x\neq 0$, then $x^Ty = x^TAx >0$, a contradiction. Thus $x=0$. This shows that LCP$(A,K,0)$ has zero as the only solution. \\
Next let $ d =a+b \in \rm int(K^*)$ with $ a> 0$ and $b \in N(A^T)$. Let $r \geq 0, r \in R(A), s=Ar+d  \in K^*$ and $r^Ts =0$. Then
\begin{center}
$0=r^Ts =r^T(Ar+d) =r^TAr + r^Ta.$
\end{center}
If $r \neq 0$, then $r^TAr \leq 0$, a contradiction to the strict copositivity of $A$ on $K$. Hence $r=0$, showing that LCP$(A,K,d)$ has zero as the only solution.
\end{proof}

The next example illustrates Theorem \ref{karcop}.

\begin{ex}\label{stcopex}
Let $A=\begin{pmatrix}
~~1 & -1 & 0\\ -1 & ~~1 & 0 \\~~0 & ~~0 & 1
\end{pmatrix}.$ Note that since $A$ has a nonnegative column, it is trivial that $K \neq \{0\}$. If $x \in R(A)$, then  $x=(\alpha, -\alpha,\beta)^T$ for some $\alpha , \beta \in \mathbb{R}$. So, if $x \in K$, then $x=(0,0,\beta)^T$, with $\beta \geq 0$. $x^TAx = \beta ^2 \geq 0$ and is strictly positive if $x \neq 0$. Hence $A$ is strictly copositive on $K$ and so it is a Karamardian matrix. Observe that, since $A$ is a $Z$-matrix, being singular, it is not a $Q$-matrix.
\end{ex}


Motivated by Proposition \ref{algo2}, let us present a procedure by which one could construct a Karamardian matrix of any order, starting from an invertible Karamardian matrix. Recall that an invertible Karamardian matrix is a $Q$-matrix. Hence the following procedure helps construct $Q$-matrices of any order.

\begin{pro}\label{karconstruct}
Let $A \in \mathbb{R}^{n \times n}$ be an invertible Karamardian matrix. Let $u \in \mathbb{R}^n$ be such that $u \geq 0$ and $\alpha >0$ be such that $\alpha \neq u^TA^{-1}u$. Then the matrix $B \in \mathbb{R}^{(n+1) \times (n+1)}$ defined by $B=\begin{pmatrix}
A & u \\ \ \ u^T & \alpha
\end{pmatrix}$ is a Karamardian matrix.
\end{pro}
\begin{proof}
First, observe that since $\alpha \neq u^TA^{-1}u$, the matrix $B$ is invertible. Let $x$ be such that $x \geq 0,~Bx \geq 0$ and $x^TBx =0$. We must show that $x=0$. Now, there exists $z \in \mathbb{R}^n_+$ and $\beta \geq 0$ such that $x:=(z^T,\beta)^T$. Now $Bx=(Az+\beta u, u^Tz+\alpha \beta) \geq 0$. We have $0=x^TBx=z^T(Az+\beta u)+\beta (u^Tz+\alpha \beta)$. Since both the terms on the right hand side are nonnegative, we have $\beta (u^Tz+\alpha \beta)=0$. If $\beta >0$, we get $u^Tz=-\alpha \beta <0$, a contradiction, since $u\geq 0$ and $z\geq 0$. Hence $\beta =0$. Now we have $Bx=(Az, u^Tz)\geq 0$ and $z^TAz=0$, i.e., $z\geq 0$, $Az\geq 0$ and $z^TAz=0$. Since $A$ is an invertible Karamardian matrix, $z=0$ and hence $x=0$ proving that the homogeneous problem has zero as the only solution.

The second part has a similar argument and is given for the sake of completeness. Note that as $A$ is Karamardian, there exists $ d > 0$, such that zero is the only vector $y$ that satisfies the conditions $y \geq 0, Ay + d \geq 0$ and $y^T(Ay+d) =0$. Define $q=(d^T,1)^T$. Then $q > 0$. Let $w \geq 0, Bw +q \geq 0$ and $w^T(Bw+q) =0$. There exists $s \in \mathbb{R}^n_+$ and $\mu \geq 0$ such that $w:=(s^T, \mu)^T$. Now $Bw+q=((As+\mu u+d)^T, u^Ts+\alpha \mu +1)^T\geq 0$ and $0=w^TBw=s^T(As+\mu u +d) +\mu (u^Ts+\alpha \mu +1)$. Again since both the terms on the right hand side are nonnegative, $\mu (u^Ts+\alpha \mu +1)=0$. By a similar argument as before, we get $\mu =0$. Thus we have $Bw+q=((As +d)^T, u^Ts)^T\geq 0$. Now, we have $s\geq 0$, $As+d\geq 0$ and $s^T(As+d)=0$. Since $A$ is an invertible Karamardian matrix, LCP$(A,d)$ has zero as the only solution and hence $s=0$. Thus $w=0$ and we have $q> 0$ such that LCP$(B,q)$ has zero as the only solution. This completes the proof that $B$ is Karamardian.
\end{proof}

The next example illustrates the construction above. 
\begin{ex}
Let $A=\begin{pmatrix}
\ \ 0 & 1 \\ -1 & 1
\end{pmatrix}$. Then $A$ is an invertible Karamardian matrix. Let $u=(1,2)^T$ and let $\alpha =1 \neq -1=u^TA^{-1}u$. Define $B:=\begin{pmatrix}
\ \ 0 & 1 & 1\\ -1 & 1 & 2 \\ ~~1 & 2 & 1
\end{pmatrix}$, using the construction of Proposition \ref{karconstruct}. Let us verify that $B$ is a Karamardian matrix, independently. First, note that $B$ is invertible. Let $x\geq 0$ be such that $Bx\geq 0$ and $x^TBx =0$. We need to show that $x=0$. Set $x=(y^T,\beta)^T$. Denoting $y=(y_1,y_2)^T$, we have from the third complementarity condition that $\beta (y_1+2y_2+\beta)=0$. If $\beta >0$, we get a contradiction and hence $\beta=0$. Since $\beta =0$, from the first two complementarity conditions, we get $y_1=y_2=0$ and hence $x=0$.

Let us turn to the non-homogeneous problem. Note that, $d=(1,1)^T$ satisfies the requirement that LCP$(A,K,d)$ has zero as the only solution. Using this $d$, we set $q=(d^T,1)=(1,1,1)^T$ so that $q > 0$. Let $z\geq 0$ be such that  $Bz+q\geq 0$ and $z^T(Bz+q) =0$. Set $z=(r^T,\gamma)^T$ and $r=(r_1,r_2)^T$. Then the nonnegativity and the complementarity constraints give us $r_1=0$. Immediately, we get $r_2=0$ and hence $\gamma =0$. Thus $z=0$. Hence $B$ is Karamardian.
\end{ex}

\subsection{The case of $2 \times 2$ matrices : Complete classification}\label{order2kar}
In this section, we obtain the complete description of all $2 \times 2$ Karamardian matrices. First, we look at the singular case. Note that if $A = uv^T$, then tr$(A) = u^T v$ and so we can restate Theorem \ref{rankonekar}, as follows:

\begin{thm} \label{Karamardian rank 1 thm involving trace}
Let $A \in \mathbb{R}^{n \times n}$ be a rank 1 matrix. Then $A$ is Karamardian if and only if $A$ has a nonnegative nonzero vector in its column space and tr$(A) > 0$. 
\end{thm}

This result enables us to characterize the $2 \times 2$ singular Karamardian matrices as follows.

\begin{thm}\label{2x2 singular Kar theorem}
\begin{enumerate}
    \item[(1)] The matrix 
\begin{align*}
A = \begin{pmatrix}
0 & 0 \\
\gamma & \delta
\end{pmatrix}
\end{align*}
is Karamardian if and only if $\delta > 0$. Also,
\begin{align*}
A = \begin{pmatrix}
\alpha & \beta \\
0 & 0
\end{pmatrix}
\end{align*}
is Karamardian if and only if $\alpha > 0$. 
\item[(2)]
The matrix
\begin{align*}
A = \begin{pmatrix}
\alpha & 0 \\
\gamma & 0
\end{pmatrix}
\end{align*}
is Karamardian if and only if $\alpha > 0$ and $\gamma \geq 0$. Also,
\begin{align*}
A = \begin{pmatrix}
0 & \beta \\
0 & \delta
\end{pmatrix}
\end{align*}
is Karamardian if and only if $\beta \geq 0$ and $\delta > 0$.
\item[(3)] Suppose 
\begin{align*}
A = \begin{pmatrix}
\alpha & \beta \\
\gamma & \delta
\end{pmatrix}
\end{align*}
where $\alpha, \beta, \gamma, \delta \neq 0$ and $\det A = 0$. Then $A$ is Karamardian if and only if tr$(A)$ $> 0$ and one of the columns of $A$ is positive.  
\end{enumerate}
\end{thm}

Now, we wish to characterize the nonsingular $2 \times 2$ Karamardian matrices. Recall that when $A$ is nonsingular, if $A$ is Karamardian then $A$ is a $Q$-matrix. First, we consider the case when $A$ is a diagonal matrix.

\begin{thm}
Let $A = \text{diag}(\alpha, \beta)$. Then $A$ is not Karamardian if and only if either (1) $\alpha < 0$ or $\beta < 0$, or (2) $A$ is singular with tr(A) $\leq 0$. 
\end{thm}
\begin{proof}
Let $A = \text{diag}(\alpha,\beta)$. If $A$ is singular, then from Theorem \ref{2x2 singular Kar theorem}, we must have tr$(A) \leq 0$, and so (2) holds. Now suppose that $A$ is nonsingular. Then $\alpha, \beta \neq 0$. If (1) does not hold, then $A$ is a nonnegative matrix with positive diagonal entries. Thus, by Theorem \ref{nonneg to be kar}, $A$ is Karamardian.

Let us consider the converse. If (2) holds, then by Theorem \ref{Karamardian rank 1 thm involving trace}, one concludes that $A$ is not Karamardian. On the other hand, suppose that (1) holds. Set $q=(-1,-1)^T$ and consider LCP$(A,q)$. If $(x_1,x_2)^T$ is a solution to LCP$(A,q)$, then clearly $x_1,x_2\neq 0$. However, since either $\alpha <0$ or $\beta <0$, LCP$(A,q)$ does not have a solution. Thus $A$ is not a $Q$-matrix and hence $A$ is not Karamardian.
\end{proof}

\begin{cor}
Let $A = diag(\alpha, \beta)$ where $\alpha, \beta \neq 0$ (so that $A$ is nonsingular). Then $A$ is Karamardian if and only if $\alpha > 0$ and $\beta > 0$.
\end{cor}

\begin{thm}
For any $\beta, \gamma \in \mathbb{R}$, the matrix
\begin{align*}
    A = \begin{pmatrix}
    0 & \beta \\
    \gamma & 0
    \end{pmatrix}
\end{align*}
is not Karamardian.
\end{thm}
\begin{proof}
Note that the result in the singular case follows from Theorem \ref{Karamardian rank 1 thm involving trace}. Now suppose $\beta, \gamma \neq 0$. We will consider the following cases: (1) $\beta, \gamma > 0$, (2) $\beta > 0$ and $\gamma < 0$, and (3) $\beta,\gamma < 0$ (the case where $\beta < 0, \gamma > 0$ follows from Theorem \ref{permutation similarity thm}). First, consider case (1). In this case, $A$ is an invertible nonnegative matrix. Since $A$ has a zero diagonal entry, $A$ is not a $Q$-matrix and hence not a Karamardian matrix. For cases (2) and (3), consider LCP$(A,q)$, where $q=(-1,-1)^T$. If $(x_1,x_2)^T$ is  a solution to LCP$(A,q)$, then clearly $x_1,x_2\neq 0$. But since either $\beta<0$ or $\gamma <0$ in cases (2) and (3), LCP$(A,q)$ does not have a solution.
\end{proof}

Now, we consider the case where $A$ is a nonsingular triangular matrix. 
\begin{thm} \label{nonsingular upper triangular Kar matrix theorem}
Let $A = \begin{pmatrix}
\alpha & \beta \\
0 & \delta
\end{pmatrix}$ or $A = \begin{pmatrix}
\delta & 0 \\
\beta & \alpha
\end{pmatrix}$ be nonsingular. Then $A$ is Karamardian if and only if $\alpha, \delta > 0$. 
\end{thm}
\begin{proof}
We will show the result for the case where $A = \begin{pmatrix}
\alpha & \beta \\
0 & \delta
\end{pmatrix}$. The case where $A = \begin{pmatrix}
\delta & 0 \\
\beta & \alpha
\end{pmatrix}$ follows from Theorem \ref{permutation similarity thm}.
Suppose $\alpha, \delta > 0$. Let $x=(x_1,x_2)^T$ be a solution of LCP$(A,0)$. Then $Ax=(\alpha x_1+\beta x_2, \delta x_2)^T\geq 0$. From $x^TAx=0$ we have $x_1(\alpha x_1+\beta x_2)=0=x_2(\delta x_2)$. Since $\delta >0$, we get $x_2=0$. This implies $x_1=0$ since $\alpha >0$. Thus $x=0$ is the only solution to LCP$(A,0)$. Similarly, we can show that LCP$(A,q)$ for $q=(1,1)^T$ has zero as the only solution.

For the converse, suppose that
\begin{align*}
    A = \begin{pmatrix}
    \alpha & \beta \\
    0 & \delta
    \end{pmatrix}
    \end{align*}
(is nonsingular and) is Karamardian. Then $A$ is a $Q$-matrix and hence LCP$(A,q)$ has a solution for all $q$. Let $(x_1,x_2)^T$ be a solution for LCP$(A,q)$, for $q=(0,-1)$. Then, $x_1(\alpha x_1+\beta x_2)=0=x_2(\delta x_2-1)$. $x_2=0$ forces $x_1=0$. This gives a contradiction, since $q\ngeq 0$. Thus $x_2=\frac{1}{\delta}$. This forces $\delta>0$. Similarly, by taking $q=(-1,0)^T$, one can show that $\alpha >0$.
\end{proof}

We also have the following.
\begin{thm}\label{one diag entry zero Kar thm}
Let $A = \begin{pmatrix}
0 & \beta \\
\gamma & \delta
\end{pmatrix}$ or $A = \begin{pmatrix}
\delta & \gamma \\
\beta & 0
\end{pmatrix}$ where $\beta,\gamma,\delta \neq 0$. Then $A$ is Karamardian if and only if $\gamma<0$ and $\beta, \delta>0$.
\end{thm}
\begin{proof}
Let $A=\begin{pmatrix}
0 & \beta \\ \gamma & \delta
\end{pmatrix}$. Since $\beta, \gamma, \delta \neq 0$, $A$ is invertible. First, we shall show that if $\gamma <0$ and $\beta, \delta >0$, then $A$ is Karamardian. \\
Let $x=(x_1,x_2)^T$ be a solution to LCP$(A,0)$. $Ax=(\beta x_2, \gamma x_1+\delta x_2)^T \geq 0$. $x^TAx=0$ gives $x_1(\beta x_2)=0=x_2(\gamma x_1+\delta x_2)$. From the first equation, one has $x_1x_2=0$ and so, from the second equation one obtains $x_2=0$. If $x_1\neq 0$, then since $\gamma <0$, we have $(Ax)_2<0$, a contradiction. Thus $x=0$ is the only solution to LCP$(A,0)$. Similarly, we can show that $x=0$ is the only solution to LCP$(A,q)$, where $q=(1,1)^T$.\\
Conversely, let $A=\begin{pmatrix}
0 & \beta \\ \gamma & \delta
\end{pmatrix}$ be a Karamardian matrix. That is, $A$ is a $Q$-matrix. We wish to show that $\beta,\delta>0$ and $\gamma <0$. Recall that a matrix with a negative row is not a $Q$-matrix (and so a nonpositive matrix is not a $Q$-matrix). Thus, at least one of $\beta, \gamma$ or $\delta$ is positive and $\gamma$ and $\delta$ are of opposite signs. Suppose $\gamma >0$. Then $\delta < 0.$ If $\beta <0$, then one can check that LCP$(A,q)$ for $q=(-1,-1)^T$ has no solution, a contradiction. Similarly if $\beta >0$ and $\gamma <0$, then LCP$(A,q)$ for $q=(1,-1)^T$ has no solution. For, in this case if $x=(x_1,x_2)^T$ is a solution, then $x_1=0$ and the second inequality becomes $\delta x_2 \geq 1,$ a contradiction again, since $\delta <0$. Thus, $\gamma <0$. Hence, $\delta >0$. If $\beta<0$, then $A$ is a $Z$-matrix. Since $A$ is a $Q$-matrix, it must be a $P$-matrix, which, however yields a contradiction since a diagonal entry of $A$ is zero. Thus $\beta >0$. We have shown that, if $A$ is a Karamardian matrix, then $\gamma <0$ and $\beta, \delta>0$.

The proof for the case when $A = \begin{pmatrix}
\delta & \gamma \\
\beta & 0
\end{pmatrix}$ follows similarly. 
\end{proof}

Next, we classify $2\times 2$ nonsingular Karamardian matrices having exactly two negative and two positive entries.

\begin{thm} \label{presummary1}
Let $\alpha,\beta,\gamma,\delta > 0$.
\begin{enumerate}
\item[(1)] If $A$ is in the form
\begin{align*}
    A = \begin{pmatrix*}[r]
    -\alpha & \beta \\
    \gamma & -\delta
    \end{pmatrix*}
    \end{align*}
    then $A$ is not Karamardian.
\item[(2)] If $A$ is in the form
\begin{align*}
    A = \begin{pmatrix*}[r]
    \alpha & -\beta \\
    -\gamma & \delta
    \end{pmatrix*}.
\end{align*}
then $A$ is Karamardian if and only if $\det A > 0$.
\item[(3)] If $A$ is in the form 
$A = \begin{pmatrix*}[r]
-\alpha & \beta \\
-\gamma & \delta
\end{pmatrix*}$
or 
$A = \begin{pmatrix*}[r]
\alpha & -\beta \\
\gamma & -\delta
\end{pmatrix*}$, then $A$ is Karamardian if and only if $\det A>0$.
\end{enumerate}
\end{thm}
\begin{proof}
(1) Let $A$ be of the form given. Let $q=(q_1,q_2)^T>0$. It can be seen that $(0, \frac{q_2}{\delta})^T$ and $(\frac{q_1}{\alpha}, 0)^T$ are two non-zero solutions of LCP$(A,q)$. Thus there is no $d\in int(K^*)$ such that LCP$(A,K,d)$ has zero as the only solution. Thus $A$ is not Karamardian.\\
(2) Suppose $A$ is a Karamardian matrix. Then since $A$ is invertible, it is a $Q$-matrix. Since $A$ is a $Z$-matrix, this holds if and only if $A$ is a $P$-matrix, and so one has $\det A >0$. Conversely, suppose that $\det A>0$. Since $A$ is a $Z$-matrix and the diagonal entries of $A$ are positive, one has that $A$ is a $P$-matrix and so it is a Karamardian matrix.\\
(3) Let $A$ be of the first form. Let $A$ be a Karamardian matrix so that $A$ is a $Q$-matrix. We shall show that $\det A>0$. Set $q=(-1,0)^T$ and let $(x_1,x_2)^T$ be a solution for LCP$(A,q)$. Then from the complementarity conditions, we have $x_1(-\alpha x_1+\beta x_2-1)=0=x_2(-\gamma x_1+\delta x_2)$. Clearly, $x_1,x_2\neq 0$. Thus we have $-\alpha x_1+\beta x_2=1$ and $\gamma x_1=\delta x_2$. Solving these equations, we have $x_1=\frac{\delta}{\beta \gamma -\alpha \delta}$ and $x_2=\frac{\gamma}{\beta \gamma -\alpha \delta}$. In particular, $0<x_1=\frac{\delta}{\beta \gamma -\alpha \delta}$. Since $\delta >0$, $-\det A=\beta \gamma-\alpha \delta >0$, showing that $\det A<0$.\\
To prove the converse, we shall show that LCP$(A,0)$ and LCP$(A,q)$ for some $q>0$ have zero as the only solution. Suppose $(x_1,x_2)^T$ is a non-zero solution for LCP$(A,0)$ so that, $x_1(-\alpha x_1+\beta x_2)=0=x_2(-\gamma x_1+\delta x_2)$. Since $\det A>0$, the case when both $x_1$ and $x_2$ are nonzero is ruled out. Further, if either $x_1$ or $x_2$ is zero, then $Ax\geq 0$ is contradicted. Thus $(0,0)^T$ is the only solution to LCP$(A,0)$. Now, we shall show that LCP$(A,q)$ for $q=(\alpha, \gamma-\epsilon)^T$, where $\epsilon$ is such that $\gamma -\epsilon >0$, has zero as the only solution. Suppose $x=(x_1,x_2)^T$ is a non-zero solution. Then $x_1(-\alpha x_1+\beta x_2 +\alpha)=0=x_2(-\gamma x_1+\delta x_2+\gamma -\epsilon)$. If $x_1=0$, then since $(x_1,x_2)^T$ is a non-zero solution, $x_2\neq 0$ and hence $\delta x_2=\epsilon -\gamma$. Since $\delta>0$, $x_2<0$, a contradiction. Suppose $x_2=0$, then $x_1\neq 0$ and hence $-\alpha x_1=-\alpha$. In this case, we have $(1,0)^T$ as a solution. But $(Ax+q)_2<0$ if $(x_1,x_2)^T=(1,0)^T$. Thus this case is ruled out. The only case left is when $x_1\neq 0$ and $x_2\neq 0$. Here we have $-\alpha x_1+\beta x_2=-\alpha$ and $-\gamma x_1+\delta x_2=-\gamma +\epsilon$. Solving these equations, we get $x_2=-\frac{\alpha \epsilon}{\det A}$. Since $\det A>0$ and $\alpha, \epsilon >0$, $x_2<0$, a contradiction. Thus $x=0$ is the only solution to LCP$(A,q)$, where $q=(\alpha, \gamma -\epsilon)^T>0$. Thus $A$ is a Karamardian matrix. The proof for the case when $A=\begin{pmatrix}
\alpha & -\beta \\
\gamma & -\delta
\end{pmatrix}$, follows similarly.
\end{proof}

\begin{thm} \label{presummary2}
Suppose that $A\in \mathbb{R}^{2\times 2}$ has exactly three positive entries and one negative entry. Then $A$ is Karamardian if and only if $A$ has one of the following forms.
\begin{align*}
    A = \begin{pmatrix*}[r]
    \alpha & -\beta \\
    \gamma & \delta
    \end{pmatrix*} \quad \text{ or } \quad 
      A = \begin{pmatrix*}[r]
    \alpha & \beta \\
    -\gamma & \delta
    \end{pmatrix*}
\end{align*}
where $\alpha,\beta,\gamma,\delta>0$.
\end{thm}
\begin{proof}
Let $A$ be a  Karamardian matrix with three positive entries and one negative entry. We show that $A$ assumes one of the forms given above. Proving this is equivalent to showing that if $A=\begin{pmatrix}
\alpha & \beta \\
\gamma & \delta
\end{pmatrix}$, then $\alpha,\delta >0$.
Note that $A$ is invertible and so it is a $Q$-matrix. Set $q=(-1,0)^T$ and let $(x_1,x_2)^T$ be a solution to LCP$(A,q)$. The complementarity conditions give $x_1(\alpha x_1+\beta x_2-1)=0=x_2(\gamma x_1+\delta x_2)$. $x_1=0$ forces $x_2=0$. Since $q\ngeq 0$, this is impossible and so $x_1\neq 0$. Thus $\alpha x_1+\beta x_2=1$. If $x_2=0$, then we have $x_1=\frac{1}{\alpha}$ and so $\alpha>0$. If $x_2\neq 0$, then one has $\alpha x_1+\beta x_2=1$ and $\gamma x_1+\delta x_2=0$, whose solution is given by $x_1=\frac{\delta}{\det A}$ and $x_2=-\frac{\gamma}{\det A}$. Since $x_1,x_2>0$, $\delta$ and $\gamma$ are of different signs. Since the matrix has exactly one negative entry, $\alpha >0$. By choosing $q=(0,-1)^T$ and applying an entirely similar argument, yields $\delta >0$.

If $A$ assumes any one of the forms mentioned in the statement, then the conclusion that $A$ is a Karamardian matrix follows from the fact that $A$ is a $P$-matrix.
\end{proof}

\section{Concluding Remarks}
	We have introduced the notion of a Karamardian matrix and have shown that this matrix class has many properties that are similar to the class of $Q$-matrices. The latter matrix class has been extensively studied in the theory of Linear Complementarity Problem (LCP). It is natural to inquire about the complementarity connection of Karamardian matrices. While we have included statements pointing to such relationships throughout the text (perhaps implicitly), we highlight some of those explicitly, here. Theorem \ref{pkar} identifies all $P$-matrices (and their inverses) to be Karamardian matrices, thereby establishing the first connection with LCP. This class includes positive definite matrices, $M$-matrices, inverse $M$-matrices and $H$-matrices with positive diagonal entries. Also, as a consequence, a subclass of the class of $Q$ matrices, namely the class of $Q\cap Z$-matrices are shown to be Karamardian. Corollary \ref{nonnegnonsingkar} shows that in the class of invertible matrices, a nonnegative matrix is a $Q$-matrix if and only if it is a Karamardian matrix.\\

\noindent{\bf Acknowledgements}\\
This work benefitted much by a critical query by the anonymous referee and the authors thank him/her for that as well as for the diligence and patience shown. The authors thank SPARC, MHRD, Government of India (Project no. P1303)  for funding during the preparation of the manuscript. Generous funding from the agency to the second author to visit Washington State University, Pullman, USA is acknowledged. The authors also express their thanks to M.S. Gowda, I. Jeyaraman, G.S.R. Murthy, T. Parthasarathy and J. Tao  for providing critical comments and suggestions on a draft version of this article. The second author acknowledges receiving INSPIRE fellowship (IF160717) from the Department of Science and Technology, India. \\

\end{document}